%% file: main.tex
\DeclareSymbolFont{AMSb}{U}{msb}{m}{n}
\documentclass[12pt,a4paper,leqno]{amsart}
\usepackage[utf8]{inputenc}
\usepackage[english]{babel}
\usepackage[a4paper,top=3cm,bottom=3cm,
           left=3cm,right=3cm,marginparwidth=60pt]{geometry}
\usepackage{indentfirst} 
\usepackage{graphicx}
\usepackage{amsmath,amsfonts,amssymb,amsthm,calligra,mathrsfs}
\usepackage{latexsym,palatino}
\usepackage{stmaryrd}
\usepackage{csquotes} 
\usepackage{tikz-cd} 
\usepackage{graphics} 
\usepackage[colorlinks,bookmarks]{hyperref} 
\usepackage{cleveref} 
\hypersetup{colorlinks, citecolor=violet, filecolor=black, linkcolor=blue, urlcolor=black}

\input{macros.tex}

\title{From formal smoothings to geometric smoothings}
\author{Alessandro Nobile}
\date{}

\begin{document}

\begin{abstract}
Let $X$ be a projective, equidimensional, singular scheme over an algebraically closed field. Then the existence of a  geometric smoothing (i.e. a family of deformations of $X$ over a smooth base curve whose generic fibre is smooth) implies the existence of a formal smoothing as defined by Tziolas.
In this paper we address the reverse question giving sufficient conditions on $X$ that guarantee the converse, i.e. formal smoothability implies geometric smoothability.
This is useful in light of Tziolas' results giving sufficient criteria for the existence of formal smoothings.
\end{abstract}

\maketitle
{\hypersetup{linkcolor=black}
\tableofcontents}

\section{Introduction}
Let $X$ be a proper $k$-scheme of finite type over an algebraically closed field $k$ of characteristic $0$. A \emph{geometric smoothing} of $X$ is a Cartesian diagram
\begin{equation}
\begin{tikzcd}
X \arrow[d]\arrow[r, hook] & \mathcal{X} \arrow[d, "p"]\\
\Spec k \arrow[r, hook, "c"] & C
\end{tikzcd}
\end{equation}
where $C$ is a smooth curve, $c\in C$ is a closed point and $p$ is a flat, proper morphism such that $p^{-1}(\eta_C)=:\mathcal{X}_{\text{gen}}$ is smooth, where $\eta_C$ is the generic point of $C$. We say that $X$ is \emph{geometrically smoothable} if it has a geometric smoothing. Following \cite[Definition~11.6]{tziolas2010smoothings}, we define a \emph{formal smoothing} of $X$ to be a formal deformation
\[
\begin{tikzcd}
X \arrow[d] \arrow[r, hook] & \mathfrak{X} \arrow[d, "\mathfrak{p}"]\\
\Spf k \arrow[r, hook] & \Spf k\llbracket t\rrbracket
\end{tikzcd}
\]
such that there exists a $b\in\mathbb{N}$ with $\mathfrak{I}^b\subset\Fitt_{\dim X}(\Omega^1_{\mathfrak{X}/\Spf k\llbracket t\rrbracket})$, where $\mathfrak{I}$ is an ideal of definition of $\mathfrak{X}$ and $\Fitt_{a}(\Omega^1_{\mathfrak{X}/\Spf k\llbracket t\rrbracket})$ is the $a^{\text{th}}$ Fitting sheaf of ideals (see \cite[\href{https://stacks.math.columbia.edu/tag/0CZ3}{Tag 0CZ3}]{stacks-project}). We say that $X$ is \emph{formally smoothable} if it admits a formal smoothing. 
Note that if $X$ is smooth then it is geometrically (hence formally) smoothable. Furthermore, Tziolas proved that geometrical smoothability implies formal smoothability. The main result of this paper is the following:

\begin{thm}\label{theorem: A}
If $X$ is a projective, equidimensional and singular scheme over $k$ such that one of the following assumptions hold:
\begin{enumerate}
\item $\coho^2(X,\mathcal{O}_X)=0$,
\item if $X$ Gorenstein, then either the dualising sheaf $\omega_X$ or its dual $\omega_X^{\vee}$ is ample,
\end{enumerate}
then the formal smoothability of $X$ is equivalent to its geometrical smoothability.
\end{thm}

The above theorem also extends Grothendieck's algebraisation theorem, see \cite[\href{https://stacks.math.columbia.edu/tag/089A}{Tag 089A}]{stacks-project}, since we have found a way to enlarge the parameter space from the spectrum of a local complete k-algebra to an affine curve.

\subsection{Motivation}
This result is motivated by the study of moduli spaces of surfaces of general type and their higher-dimensional analogues.

Moduli spaces of surfaces of general type are well studied and it is known that stable surfaces lie within the compactification of these moduli spaces.

A \emph{stable surface}, see \cite{Kollarbook}, is a proper two-dimensional reduced connected scheme satisfying one local and one global condition. The local condition bounds the badness of singularities that such surfaces can have, requiring them to be semi-log-canonical (see \cite[Definition~1.40]{Kollarbook}).
The global condition requires the dualising sheaf to be ample.

Since stable surfaces appear as points on the boundary of the moduli space of surfaces of general type, it is of great interest to understand which stable surfaces are geometrically smoothable.

In order to understand which surfaces can be smoothed, it is important to know which singularities among the semi-log-canonical ones can be smoothed. The class of such singularities is very broad since it admits both isolated and non-isolated singularities. If $X$ has isolated singularities, it is known \cite[Proposition~2.4.6]{sernesi2007deformations} that $\coho^2(X,\mathcal{T}_X)$ is an obstruction space to the extension of local smoothings to global ones.

The study of non-isolated singularities is not so easy. In \cite{Persson1983SomeEO}, they gave examples of non-smoothable singularities with normal crossing divisors, showing that not all non-isolated singularities are smoothable. Another difficulty that one has to face studying non-isolated singularities is that the Schlessinger's cotangent sheaf $\mathcal{T}^1$ (and its higher analogues), which is a sheaf supported on the singular locus, is difficult to describe and sometimes not finite dimensional, as shown in \cite{Fantechi2017OnTR}. An application to Godeaux surfaces of \Cref{theorem: A} is given in \cite{fantechi2021smoothing}.

\subsection{Structure of the paper}
This paper is an expository article on formal schemes, formal deformation and smoothing. It organized in four sections: in the first one it is collected an introduction to formal schemes, following the treatment of Illusie in \cite{FGAIll} and of Alonso, Jerem\'ias and P\'erez in \cite{tlr1} and \cite{tlr2}. The second section contains a discussion on formal deformation theory with, what we hope, a clear treatment on the differences between the various type of definition of deformations. We decide to add this information in order fix the terminology and better clarify what is the main point of this article. This section ends with a discussion of two different notions of smoothing of a scheme; in particular, in there we motivate the definition of formal smoothing as given by Tziolas in \cite{tziolas2010smoothings}. The third section is an overview of the Gorenstein condition and its behaviour under deformation, mostly following \cite{stacks-project}. Since we were not able to find a reference in the literature, in this section we include a proof of a classical result on good behaviour of the Gorenstein property under infinitesimal deformations. The fourth and last part contains the main result, its proof and an example of application to a real moduli problem.

\subsection{Conventions}\label{conventions}
All schemes are defined over an algebraically closed field $k$ of characteristic $0$. We will assume that all schemes will be of finite type and separated and we will denote by $\text{FTS}_k$ (or simply by FTS) the category whose objects are separated, finite type $k$-schemes and whose morphisms are morphisms of $k$-schemes.

\subsection{Acknowledgment}
I would like to thank the algebraic geometry group at SISSA for useful mathematical discussions and precious suggestions. A very special thanks is due to my PhD advisor, prof. Barbara Fantechi, for her constant patience, support and precious advices.\\
I would also like to thank the algebraic geometry group at Universt\'e du Luxembourg.

\section{Locally Noetherian formal schemes}
We recall for the reader's convenience some basic results on formal schemes. We follow Illusie's and Grothendieck's language and presentation in \cite{FGAIll} and \cite{MR0217083} respectively. At some points we will also refer to articles \cite{tlr1} and \cite{tlr2} by Alonso, Jerem\'ias and P\'erez.

\subsection{The category of locally Noetherian formal schemes}
\begin{definition}\label{def: adic ring}
An \emph{adic} (or \emph{$I$-adic}) \emph{Noetherian ring} is a topological Noetherian ring $A$ that admits an ideal $I$, called an \emph{ideal of definition}, such that
\begin{itemize}
\item $\{I^n\}_{n\in\mathbb{N}}$ is a fundamental system of neighbourhoods of $0$ in $A$;
\item the topology induced on $A$ turns $A$ into a separated and complete topological space.
\end{itemize}
\end{definition}

In general an ideal of definition is not unique. Indeed for another ideal $J$ to be an ideal of definition it is necessary and sufficient that there are two non-negative integers $n,m$ such that $J\supset I^m\supset J^n$.

We remark that $A$ is an $I$-adic Noetherian ring if and only if $A=\varprojlim_n A/I^{n}=:\hat{A}$, where $\hat{A}$ denotes the formal completion of $A$ along the ideal $I$.

Examples of adic Noetherian rings are the \emph{ring of formal power series} and the \emph{ring of restricted power series}, see \cite[Example 1.6]{tlr1}, in the following denoted respectively by $k\llbracket t\rrbracket$ and $A\{T_1,\dots,T_n\}$, with $A$ an $I$-adic Noetherian ring and $t$, $T_1,\dots,T_n$ indeterminates.

We wish to introduce the notion of an affine formal scheme, needed for the definition of a formal scheme. If $A$ is an $I$-adic Noetherian ring $A$ and $n$ a non-negative integer, we denote by $A_n$ the quotient $A/I^{n+1}$ and by $X_n$ the affine scheme $\Spec A_n$. We then have a chain of closed subschemes
\[
X_0\subset X_1\subset\cdots\subset X_n\subset\cdots
\]
and all these subschemes have the same underlying topological space $|\Spec A/I|$.

\begin{definition}\label{def: formal affine spectrum}
Let $A$ be an adic Noetherian ring with $I$ an ideal of definition. The \emph{affine formal spectrum of $A$} is the topologically ringed space $(\Spf A, \mathcal{O}_{\Spf A})$ where
\begin{itemize}
\item the topological space is
\[
\Spf A:=\{\mathfrak{p}\in\Spec A\colon I\subset \mathfrak{p}\}
\]
which is naturally homeomorphic to $|\Spec A/I |$. Equivalently, we could have defined $\Spf A$ to be the topological space made by open primes ideals of $A$;
\item the structure sheaf is
\[
\begin{aligned}
\mathcal{O}_{\Spf A}:=\varprojlim_n \mathcal{O}_{X_n}
\end{aligned}
\]
\end{itemize}
and is a sheaf of topological rings. Its topology is given by
\[
\Gamma(U,\mathcal{O}_{\Spf A})=\varprojlim_n\Gamma(U,\mathcal{O}_{X_n})
\]
for every open subset $U$ of $\Spf A$, where $\Gamma(U,\mathcal{O}_{X_n})$ has the discrete topology.
\end{definition}

The definition above does not depend on the ideal of definition. Indeed if a prime ideal $\mathfrak{p}$ of $A$ contains $I$ it also contains all of its powers, in particular it contains $I^m$ and hence $J^n$. Since $\mathfrak{p}$ is prime, it follows that it contains also $J$.

Since the topology of $\Spf A$ admits a base of neighbourhoods made by quasi-compact open subsets, it is enough to require that for every quasi-compact open subset $U$ of $\Spf A$,
\[
\Gamma(U,\mathcal{O}_{\Spf A})=\varprojlim_n\Gamma(U,\mathcal{O}_{X_n}),
\]
where $\Gamma(U,\mathcal{O}_{X_n})$ has the discrete topology (see \cite[(\textbf{1}.10.1.1)]{MR0217083}).

\begin{remark}\label{remark: canonical inclusion}
For an $I$-adic Noetherian ring $A$, the canonical morphism $A\to\hat{A}$ is an isomorphism and it induces a morphism of ringed spaces from $\Spf\hat{A}=\Spf A$ to $\Spec A$.
\end{remark}

We can now define what are affine Noetherian formal schemes and locally Noetherian formal schemes.

\begin{definition}\label{def: affine Noetherian formal scheme}
An \emph{affine Noetherian formal scheme} is a topologically ringed space isomorphic to an affine formal spectrum as in \Cref{def: formal affine spectrum}.
\end{definition}

\begin{definition}\label{def: formal scheme}
A \emph{locally Noetherian formal scheme} is a topologically ringed space $(\mathfrak{X}, \mathcal{O}_{\mathfrak{X}})$ such that every point has an open neighbourhood which is isomorphic to an affine Noetherian formal scheme.

A \emph{Noetherian formal scheme} is a quasi-compact locally Noetherian formal schemes.
\end{definition}

Since affine formal schemes are locally topologically ringed spaces, locally Noetherian formal schemes are locally topologically ringed spaces.

As in the classical case, we denote the locally Noetherian formal scheme $(\mathfrak{X}, \mathcal{O}_{\mathfrak{X}})$ by $\mathfrak{X}$.

\begin{notation}
For the rest of the article we will abbreviate ``locally Noetherian formal scheme'' by LNFS.
\end{notation}

Example of locally Noetherian formal schemes, which are in particular affine Noetherian formal schemes, are $\Spf k\llbracket t\rrbracket$ and $\Spf A\{T_1,\dots,T_n\}$. In what follows, we will denote the formal scheme $\Spf A\{T_1,\dots,T_n\}$ by $\mathbb{A}^{n}_{\Spf A}$ and we will call it the formal affine $n$-space. Observe that the underlying topological space of $\mathbb{A}^{n}_{\Spf A}$ is $\Spec\left( (A/I)[T_1,\dots, T_n]\right)$.

\begin{notation}
In what follows we will denote $\Spf k\llbracket t\rrbracket$ by $\mathfrak{S}$ and, for every non-negative integer $n$,
$S_n$ will denote $\Spec\frac{k\llbracket t\rrbracket}{(t^{n+1})}=\Spec\frac{k[ t]}{(t^{n+1})}$.
\end{notation}

Now we define morphisms between LNFSs.

\begin{definition}\label{def: morphism of formal schemes}
Let $\mathfrak{X}$ and $\mathfrak{Y}$ be two LNFSs. A \emph{morphism of LNFSs} is a morphism $\mathfrak{f}\colon\mathfrak{X}\to\mathfrak{Y}$ of locally ringed spaces such that for every open subset $\mathfrak{V}$ of $\mathfrak{Y}$ the induced map
\[
\Gamma(\mathfrak{V}, \mathcal{O}_{\mathfrak{Y}})\to\Gamma(\mathfrak{f}^{-1}(\mathfrak{V}), \mathcal{O}_{\mathfrak{X}})
\]
is continuous.
\end{definition}

As in the classical case of schemes, there is an equivalence of categories between adic Noetherian rings and affine Noetherian formal schemes, for more see \cite[(\textbf{1}.10.2.2)]{MR0217083}. Furthermore, the classical adjunction holds also in the case of LNFSs.

\begin{proposition}[{\cite[(\textbf{1}.10.4.6)]{MR0217083}}]
Let $\mathfrak{X}$ be a LNFS and let $A$ be a Noetherian adic ring. Then there is a natural bijection between morphisms of locally Noetherian formal schemes from $\mathfrak{X}$ to $\Spf A$ and continuous ring homomorphisms from $A$ to $\Gamma(\mathfrak{X}, \mathcal{O}_{\mathfrak{X}})$.
\end{proposition}

As a further example of formal schemes, we can consider the \emph{completion of a scheme along a closed subscheme}.

\begin{example}\label{example: formal completion along a closed subscheme}
Suppose that $X$ is a locally Noetherian scheme and consider a closed subscheme $Y$ of $X$ with sheaf of ideals given by $\mathcal{I}$. Then we can consider the schemes $X_n:=(Y, \mathcal{O}_{X}/\mathcal{I}^{n+1})$, for every $n\in\mathbb{N}$, which gives rise to the sequence of thickenings
\[
X_0\hookrightarrow X_1\hookrightarrow\cdots X_n\hookrightarrow\cdots
\]
Taking now the colimit we get a LNFS, denoted by $\hat{X}_{/Y}$ and called \emph{the formal completion of $X$ along $Y$}.
\end{example}

We point out that, if $Y=X$, then $\hat{X}_{/X}=X$. Therefore the category of LNFSs contains the category of Noetherian schemes.

However, Hironaka and Matsumura in \cite[Theorem~(5.3.3) page 81]{hironaka1968formal} and independently Hartshorne in \cite[Example 3.3 page 205]{hartshorne2006ample} constructed two examples showing that not all formal schemes appear as the completion of a single scheme along a closed subscheme. This consideration motivates the following definition.

\begin{definition}\label{def: algebraisable formal scheme}
A LNFS $\mathfrak{X}$ is called \emph{algebraisable} if there are a scheme $X$ and a closed subscheme $Y$ of $X$ such that $\mathfrak{X}=\hat{X}_{/Y}$.
\end{definition}

\subsection{Sheaves on LNFSs}
We now define the notion of a coherent formal sheaf on a LNFS. In the classical case of affine Noetherian schemes there is the functor $\widetilde{(-)}$ that associates to any finitely generated module its coherent sheaf. Similarly, in the formal case there is the functor $(-)^{\Delta}$ which associates to any finitely generated module its formal coherent sheaf.

Note that if $A$ is an adic Noetherian ring, then every $A$-module $M$ has an induced $I$-adic topology where a system of fundamental neighbourhoods of $0$ is given by $\{I^n\cdot M\}_{n\in\mathbb{N}}$.

\begin{notation}\label{notation: homomorphisms instead of continuous homomorphisms}
If $A$ is a Noetherian $I$-adic ring and $M$ and $N$ are finitely generated $A$-modules that are separated and complete in the induced $I$-adic topology, then, by \cite[(\textbf{0}.7.8.1)]{MR0217083} it follows that every $A$-module homomorphism is automatically continuous. Therefore, in what follows, we will write $\Hom_{A}(M,N)$ in place of $\Hom_{A-\text{cont}}(M,N)$.

Furthermore, by \cite[(\textbf{0}.7.8.2)]{MR0217083} we have a canonical isomorphism
\[
\Hom_{A}(M,N)\stackrel{\cong}{\rightarrow}\varprojlim_{n}{}{}\Hom_{\frac{A}{I^{n+1}}}\left(\frac{M}{I^{n+1}M},\frac{N}{I^{n+1}N}\right).
\]
From this we conclude that, if $A$ is a Noetherian $I$-adic ring and $M$ is a finitely generated $A$-module, then
\[
M^{\vee}:=\Hom_{A}(M, A)=\varprojlim_{n}{}{}\Hom_{\frac{A}{I^{n+1}}}\left(\frac{M}{I^{n+1}M},\frac{A}{I^{n+1}}\right)=\varprojlim_{n}\left(\frac{M}{I^{n+1}M}\right)^{\vee}.
\]
\end{notation}

\begin{definition}\label{def: Delta construction for modules}
Let $A$ be an $I$-adic Noetherian ring and let $M$ be a finitely generated $A$-module. Then we define the \emph{coherent formal sheaf} $M^{\Delta}$ on $\Spf A$ to be the completion of $\widetilde{M}$ along the ideal sheaf $\widetilde{I}$ of the closed embedding $\Spec A/I\hookrightarrow\Spec A$:
\[
M^{\Delta}:=\varprojlim_n \frac{\widetilde{M}}{\widetilde{I^n}\cdot\widetilde{M}}.
\]
\end{definition}

The functor $(-)^{\Delta}$ satisfies similar properties of the functor $\widetilde{(-)}$, for more see \cite[(\textbf{1}.10.10.2)]{MR0217083}.

\begin{definition}\label{def: ideal sheaf of definition} 
An \emph{ideal of definition} of a LNFS $\mathfrak{X}$ is a formal coherent sheaf of ideals $\mathfrak{I}$ of $\mathcal{O}_{\mathfrak{X}}$ such that for any point $x\in\mathfrak{X}$ there exists a formal affine neighbourhood $\Spf A$ of $x$ in $\mathfrak{X}$ and there exists an ideal of definition $I$ of $A$ such that $\mathfrak{I}|_{\Spf A}=I^{\Delta}$.
\end{definition}

A formal coherent sheaf $\mathfrak{I}$ on a LNFS $\mathfrak{X}$ is an ideal of definition if and only if $(\mathfrak{X},\frac{\mathcal{O}_{\mathfrak{X}}}{\mathfrak{I}})$ is a scheme.
Actually, for any LNFS $\mathfrak{X}$ there exists a maximal ideal of definition $\mathfrak{I}$ which is the unique ideal of definition such that $(\mathfrak{X},\frac{\mathcal{O}_{\mathfrak{X}}}{\mathfrak{I}})$ is a reduced scheme.
In a Noetherian formal scheme the ideal of definition is not unique; indeed, any other formal coherent sheaf of ideals $\mathfrak{J}$ on the LNFS $\mathfrak{X}$ is an ideal of definition if and only if there are positive integers $m,n$ such that the chain of inclusions
$\mathfrak{J}\supset\mathfrak{I}^{m}\supset\mathfrak{J}^n$ holds.

\begin{remark}\label{rem: definition of formal scheme as collection of infinitesimal neighbourhoods}
As in the affine formal case, it is also possible to define LNFSs as a collection of all of their infinitesimal neighbourhoods (or thickenings).

More precisely, let $\mathfrak{X}$ be a LNFS and let $\mathfrak{I}$ be an ideal of definition. For every $n\in\mathbb{N}$, define $(X_n,\mathcal{O}_{X_n})$ to be the ringed space $(|\mathfrak{X}|, \mathcal{O}_{\mathfrak{X}}/\mathfrak{I}^{n+1})$ which is a locally Noetherian scheme. This induces a sequence of closed embeddings
\[
X_0\hookrightarrow X_1\hookrightarrow X_2\hookrightarrow \cdots\hookrightarrow X_n\hookrightarrow \cdots
\]
whose ideals of definition are nilpotent and all the maps on the underlying topological spaces are the identity. Then $\mathfrak{X}$ can be recovered from the above sequence of thickenings by passing through the direct limit in the category of locally Noetherian topologically ringed spaces, i.e.
\[
\mathfrak{X}=\varinjlim_n X_n.
\]
In particular there are natural morphisms of ringed spaces
\[
\alpha_n\colon X_n\to\mathfrak{X},
\]
where $\alpha_n$ is the identity on the underlying topological space and the map of sheaves of topological rings is just the quotient map
\[
\alpha_n^{\natural}\colon\mathcal{O}_{\mathfrak{X}}\to\mathcal{O}_{X_n}=\frac{\mathcal{O}_{\mathfrak{X}}}{\mathfrak{I}^{n+1}}.
\]
Conversely, see \cite[(\textbf{1}.10.6.3)]{MR0217083}, given a collection $\{X_n\}_{n\in\mathbb{N}}$ of locally Noetherian schemes satisfying:
\begin{itemize}
\item[(i)] for every $n$, there are morphisms of schemes $\psi_{n+1,n}\colon X_n\to X_{n+1}$ such that they are homeomorphisms on the underlying topological spaces and induce surjective morphisms of sheaves $\mathcal{O}_{X_{n+1}}\to\mathcal{O}_{X_{n}}$;
\item[(ii)] if $\mathcal{J}_n:=\ker(\mathcal{O}_{X_n}\to\mathcal{O}_{X_0})$, then $\ker(\mathcal{O}_{X_n}\to\mathcal{O}_{X_m})=\mathcal{J}_n^{m+1}$, for $m\leq n$;
\item[(iii)] $\mathcal{J}_1\in\coh(X_0)$;
\end{itemize}
then the topologically ringed space $\mathfrak{X}:=\varinjlim_n X_n$ obtained by taking the direct limit is a LNFS. Moreover, denoting by $\mathfrak{I}:=\ker(\mathcal{O}_{\mathfrak{X}}\to\mathcal{O}_{X_0})$, then $\mathfrak{I}$ is an ideal of definition of $\mathfrak{X}$ and satisfies the following properties
\[\mathfrak{I}=\varprojlim_n\mathcal{J}_n\,\,\,\,\,\,\,\,\,\,\,\,\,\,\text{ and }\,\,\,\,\,\,\,\,\,\,\,\,\,\,\mathfrak{I}^{n+1}=\ker(\mathcal{O}_{\mathfrak{X}}\to\mathcal{O}_{X_n}).
\]
\end{remark}

\begin{definition}\label{def: coherent formal sheaf}
A \emph{coherent formal sheaf} on a LNFS $\mathfrak{X}$ is a sheaf $\mathfrak{F}$ such that, for every open Noetherian formal affine subset $\mathfrak{U}=\Spf A$ of $\mathfrak{X}$, there exists a finitely generated $A$-module $M$ with $\mathfrak{F}|_{\mathfrak{U}}=M^{\Delta}$.
\end{definition}

Next we give an interpretation of coherent formal sheaves on a LNFS as the limit of coherent sheaves on all thickenings.

\begin{remark}\label{remark: formal coherent sheaf as collection of coherent sheaves on the thickenings}
Let $\mathfrak{X}$ be a LNFS, let $\mathfrak{I}$ be an ideal of definition and let $\mathfrak{F}$ be a coherent formal sheaf of $\mathcal{O}_{\mathfrak{X}}$-modules. For every $n$, let us denote by $X_n$ the locally Noetherian scheme as defined in \Cref{rem: definition of formal scheme as collection of infinitesimal neighbourhoods}. If, for every $n$, we define
\[
\mathscr{F}_n:=\frac{\mathfrak{F}}{\mathfrak{I}^{n+1}\mathfrak{F}},
\]
then we have that $\mathscr{F}_n\in\coh(X_n)$ and we recover $\mathfrak{F}$ by considering $\varprojlim_n\mathscr{F}_n$.
Conversely, see \cite[(\textbf{1}.10.11.3)]{MR0217083}), let $\mathfrak{X}$ be a locally Noetherian scheme and $\mathfrak{I}$ and ideal of definition of $\mathfrak{X}$. Let $\{X_n\}_{n\in\mathbb{N}}$ be a collection of locally Noetherian schemes defining $\mathfrak{X}$ as in \Cref{rem: definition of formal scheme as collection of infinitesimal neighbourhoods} and, for $m\leq n$, let $\psi_{n,m}\colon X_m\to X_n$ denote the canonical maps. Suppose that for every $n\in\mathbb{N}$, $\mathscr{F}_n$ is a coherent sheaf on $X_n$ together with morphisms, for $m\leq n$
\[
\phi_{n,m}\colon\mathscr{F}_{m}\to(\psi_{n,m})_*\mathscr{F}_{n},
\]
such that for every $l\geq m\geq n$ we have $\phi_{n,m}\circ\phi_{m,l}=\phi_{n,l}$\footnote{The conditions listed here are equivalent to requiring that the system $\{\mathscr{F}_n, \phi_{m,n}\}_{n,m\in\mathbb{N}}$ be a projective system.}. Then the limit $\mathfrak{F}:=\varprojlim_n\mathscr{F}_n$ is a coherent formal sheaf on $\mathfrak{X}$.
\end{remark}

\begin{definition}\label{def: locally free formal sheaf} Let $\mathfrak{X}$ be a LNFS and $r\in\mathbb{N}$. We say that a formal coherent sheaf $\mathfrak{F}$ on $\mathfrak{X}$ is locally free of rank $r$ if for every open Noetherian affine subset $\mathfrak{U}=\Spf A$ of $\mathfrak{X}$, the finitely generated $A$-module $M$ (which exists since $\mathfrak{F}$ is coherent) is free of rank $r$.
\end{definition}

We can give an equivalent definition of formal coherent sheaf on a LNFS based on the infinitesimal thickening description of LNFSs. It is done as follows: a formal coherent sheaf $\mathfrak{F}$ on a LNFS $\mathfrak{X}$ is locally free of finite rank $r$ if each sheaf $\mathscr{F}_n:=\frac{\mathfrak{F}}{\mathfrak{I}^{n+1}\mathfrak{F}}$ is locally free of the same rank $r$, for all natural numbers $n$, where $\mathfrak{I}$ is an ideal of definition of the formal scheme $\mathfrak{X}$.

\subsection{Adic morphisms between LNFSs}
In order to give a description in terms of thickenings for morphisms of formal schemes, we need to restrict our interest to a particular kind of morphisms: the adic morphisms.

\begin{definition}\label{def: adic morphism}
A morphism $\mathfrak{f}\colon\mathfrak{X}\to\mathfrak{Y}$ of LNFSs is called an \emph{adic morphism} if there exists an ideal of definition $\mathfrak{J}$ of $\mathfrak{Y}$ such that $\mathfrak{f}^*\mathfrak{J}\cdot\mathcal{O}_{\mathfrak{X}}$ is an ideal of definition of $\mathfrak{X}$.
\end{definition}

The definition of an adic morphism does not depend on the choice of the ideal of definition; indeed one could equivalently ask that the condition $\mathfrak{f}^*\mathfrak{J}\cdot\mathcal{O}_{\mathfrak{X}}$ holds \textit{for all} ideals of definition of $\mathfrak{Y}$ (see \cite[(\textbf{1}.10.12.1)]{MR0217083}). Observe that if $\mathfrak{f}\colon\mathfrak{X}\to\mathfrak{Y}$ is an adic morphism bethween LNFSs, then the topology on $\mathcal{O}_{\mathfrak{Y}}$ determines the topology on $\mathcal{O}_{\mathfrak{X}}$.

\begin{remark}\label{rem: morphism of formal schemes as a collection of compatible morphisms of schemes}
Suppose that $\mathfrak{f}\colon\mathfrak{X}\to\mathfrak{Y}$ is an adic morphism of LNFSs and let $\mathfrak{J}$ and $\mathfrak{I}:=\mathfrak{f}^*\mathfrak{J}\cdot\mathcal{O}_{\mathfrak{X}}$ be ideals of definition of $\mathfrak{Y}$ and $\mathfrak{X}$ respectively.

Then we can consider the sequences of thickenings
\[
X_0\hookrightarrow X_1\hookrightarrow\cdots X_n\hookrightarrow\cdots
\,\,\,\,\text{ and }\,\,\,\,
Y_0\hookrightarrow Y_1\hookrightarrow\cdots Y_n\hookrightarrow\cdots
\]
as in \Cref{rem: definition of formal scheme as collection of infinitesimal neighbourhoods}.
Since the morphism was supposed to be adic, we get that for every $n\in\mathbb{N}$, $\mathfrak{f}^*(\mathfrak{J}^{n+1})\cdot\mathcal{O}_{\mathfrak{X}}=\mathfrak{I}^{n+1}$. Therefore we have induced morphisms
\[
f_n\colon X_n\to Y_n
\]
such that all the squares
\begin{equation}\label{eq: adic morphism have Cartesian squares}
\begin{tikzcd}
X_n \arrow[r, "f_n"] \arrow[d, hook] & Y_n \arrow[d, hook]\\
X_{n+1} \arrow[r, "f_{n+1}"] & Y_{n+1}
\end{tikzcd}
\end{equation}
are Cartesian. Then $\mathfrak{f}$ can be recovered by the collection of morphisms $\{f_n\}_{n\in\mathbb{N}}$ by considering the colimit, i.e. $\mathfrak{f}=\varinjlim_n f_n$.

Conversely (see \cite[(8.1.5)]{FGAIll}), any system of morphisms of locally Noetherian schemes $\{f_n\colon X_n\to Y_n\}_{n\in\mathbb{N}}$ such that all squares \cref{eq: adic morphism have Cartesian squares} are Cartesian induces an adic morphism of LNFSs by considering the colimit.
\end{remark}

\subsection{Properties of adic morphisms}
Now we introduce the notions of finite type, properness and flatness for a morphism of formal schemes.

\begin{definition}\label{def: morphism of finite type of formal scheme}
Let $\mathfrak{X}$ and $\mathfrak{Y}$ be LNFSs. A morphism $\mathfrak{f}\colon\mathfrak{X}\to\mathfrak{Y}$ is said to be \emph{of finite type} if $\mathfrak{f}$ is an adic morphism and the induced morphism $f_0\colon X_0\to Y_0$ is of finite type.
\end{definition}

\begin{definition}\label{def: proper morphism of formal schemes}
A morphism $\mathfrak{f}\colon\mathfrak{X}\to\mathfrak{Y}$ of LNFSs is \emph{proper} if it is of finite type and $f_0\colon X_0\to Y_0$ is proper.
\end{definition}

\begin{definition}\label{def: flat morphism of formal schemes}
Let $\mathfrak{f}\colon\mathfrak{X}\to\mathfrak{Y}$ be a morphism of LNFSs. We say that $\mathfrak{f}$ is \emph{flat} if it is adic and for every $x\in\mathfrak{X}$, $\mathcal{O}_{\mathfrak{X},x}$ is a flat $\mathcal{O}_{\mathfrak{Y}, \mathfrak{f}(x)}$-module.
\end{definition}

\begin{proposition}
Let $\mathfrak{f}\colon\mathfrak{X}\to\mathfrak{Y}$ be an adic morphism of LNFSs, let $\{f_n\colon X_n\to Y_n\}_{n\in\mathbb{N}}$ be a compatible collection associated to $\mathfrak{f}$ and $\mathcal{P}$ be one of the following properties of morphisms: of finite type, proper, flat. Then the following conditions are equivalent:
\begin{enumerate}
\item $\mathfrak{f}$ has $\mathcal{P}$;
\item $f_n$ has $\mathcal{P}$, for every $n\in\mathbb{N}$.
\end{enumerate}
\end{proposition}

We point out that we could have defined a flat morphism of LNFSs $\mathfrak{f}\colon\mathfrak{X}\to\mathfrak{Y}$ without assuming it to be adic. However, with that choice, we would not be able to deduce the flatness of $\mathfrak{f}$ from the flatness of all $\{f_n\}_{n\geq0}$ and vice versa. See \cite[Proposition~3.3]{tlr2} for the local criterion of flatness for formal schemes, and \cite[Example~3.2]{tlr2} gives a counter example.

We conclude the section by presenting one result needed in the proof of the main result.

\begin{theorem}[{\cite[II - Ex. 9.6(c)]{hartshorne1977algebraic}}]\label{teo: locally free sheaves on each nilpotent subscheme induce a locally free sheaf on the completion} 
Let $\mathfrak{X}$ be a LNFS, let $\mathfrak{I}$ be an ideal of definition of $\mathfrak{X}$ and, for each $n\in\mathbb{N}$, let us denote by $X_n$ the scheme $(\mathfrak{X}, \mathcal{O}_{\mathfrak{X}}/\mathfrak{I}^n)$. Suppose that, for every $n\in\mathbb{N}$, we are given invertible sheaves $\mathscr{L}_n$ on $X_n$ together with isomorphisms $\mathscr{L}_{n+1}\otimes_{\mathscr{O}_{X_{n+1}}}\mathscr{O}_{X_n}\cong\mathscr{L}_n$. Then the sheaf
\[
\mathfrak{L}:=\varprojlim_n\mathscr{L}_n
\]
is an invertible sheaf on $\mathfrak{X}$.
\end{theorem}

\section{On deformations and smoothings}
In this section we introduce various definitions of deformations of a scheme and we discuss their relationship. Then we present and explain the two different definitions of smoothing of a scheme used in this paper.

\subsection{Introducing formal deformations}
\begin{definition}\label{def: formal deformation}
Let $X$ be a scheme and let $(R, \mathfrak{m})$ be a complete local ring. A \emph{formal deformation} of $X$ over $R$ is a Cartesian diagram
\begin{equation}\label{eq: formal deformation}
\begin{tikzcd}
X \arrow[d] \arrow[r, hook] & \mathfrak{X} \arrow[d, "\mathfrak{f}"]\\
\Spf(\frac{R}{\mathfrak{m}}) \arrow[r, hook] & \Spf R
\end{tikzcd}
\end{equation}
with $\mathfrak{f}$ a flat morphism.
\end{definition}

\begin{notation}
In the future, in order to ease the notation, we will denote any deformation (either classical or formal) by its flat morphism. For example, we will refer to the formal deformation of \cref{eq: formal deformation} only by $\mathfrak{f}\colon\mathfrak{X}\to\Spf R$.
\end{notation}

As we have seen before in \Cref{rem: definition of formal scheme as collection of infinitesimal neighbourhoods}, formal schemes can be equivalently described as compatible collection of infinitesimal thickenings. A similar description can be given for formal deformations.

\begin{remark}\label{rem: formal deformations and thickenings}
Fix a formal deformation of a scheme $X$ as in \cref{eq: formal deformation} and, for any non-negative integer $n$, let us denote by $R_n$ the quotient ring $R/\mathfrak{m}^{n+1}$. Then, for any $n\geq0$, we have diagrams
\[
\begin{tikzcd}
 & \mathfrak{X} \arrow[d, "\mathfrak{f}"]\\
 \Spec R_n \arrow[r, hook] & \Spf R.
\end{tikzcd}
\]
Pulling back $\mathfrak{f}$ along the closed immersion $\Spec R_n\hookrightarrow\Spf R$, we obtain a collection of deformations $\{f_n\colon\mathcal{X}_n\to\Spec R_n\}_{n\geq0}$ of $X$ over $\Spec R_n$. Moreover, by construction, all these deformations of $X$ are compatible, i.e. for every non negative integer $n$, we have Cartesian diagrams
\[
\begin{tikzcd}
\mathcal{X}_n \arrow[r, hook] \arrow[d, "f_n"] & \mathcal{X}_{n+1} \arrow[d, "f_{n+1}"]\\
\Spec R_n \arrow[r, hook] & \Spec R_{n+1}.
\end{tikzcd}
\]
\end{remark}

The converse also holds true, as stated in the following proposition.

\begin{proposition}[{\cite[Proposition~21.1]{hartshorne2009deformation}}]\label{prop: formal deformation as compatible collection of deformations}
Let $(R,\mathfrak{m})$ be an adic local Noetherian ring with residue field $k$, let $X$ be a scheme and define and, for every non-negative integer $n$, let $R_n:=R/\mathfrak{m}^{n+1}$. Suppose that for every $n\in\mathbb{N}$ we are given a family $\{f_n\colon\mathcal{X}_{n}\to\Spec R_{n}\}_{n\geq0}$ of infinitesiaml deformations such that $\mathcal{X}_0 = X$, the morphisms $f_n$ are flat, of finite type and the following compatibility condition holds: for all $n\geq0$, the diagrams
\begin{equation}\label{eq: compatibility of deformation to induced a formal deformation}
\begin{tikzcd}
\mathcal{X}_{n} \arrow[r, hook] \arrow[d, "f_{n}"] & \mathcal{X}_{n+1} \arrow[d, "f_{n+1}"]\\
\Spec R_{n} \arrow[r, hook] & \Spec R_{n+1}
\end{tikzcd}
\end{equation}
are all Cartesian.

Then there exists a (Noetherian) formal scheme $\mathfrak{X}$, flat over $\Spf R$, such that $\mathcal{X}_n\cong\mathfrak{X}\times_{\Spf R}\Spec R_n$, for every natural number $n$.
\end{proposition}

Concluding, \Cref{rem: formal deformations and thickenings} together with \Cref{prop: formal deformation as compatible collection of deformations} imply that a formal deformation $\mathfrak{f}\colon\mathfrak{X}\to\Spf R$ is uniquely determined by a family of infinitesimal deformations $\{f_n\colon\mathcal{X}_n\to\Spec R_n\}_{n\geq0}$ satisfying the compatibility condition expressed by asking that all diagrams of \cref{eq: compatibility of deformation to induced a formal deformation} must be Cartesian.

Next we explain how to construct a formal deformation starting from a deformation over the spectrum of an algebra essentially of finite type.

\begin{remark}\label{remark: constructing formal deformation from classic deformation}
Let $X$ be a scheme, let $(A,\mathfrak{m})$ be a $k$-algebra essentially of finite type, i.e. a localisation of a $k$-algebra of finite type. Consider a deformation of $X$ over $A$
\[
\begin{tikzcd}
X \arrow[r, hook] \arrow[d] & \mathcal{X} \arrow[d, "f"]\\
\Spec k \arrow[r, hook] & \Spec A.
\end{tikzcd}
\]
Let $\hat{A}$ be the formal completion of $A$ at $\mathfrak{m}$; for every $n\geq0$, define $A_n$ to be the quotient ring $A/\mathfrak{m}^{n+1}$ and note that we have canonical isomorphisms $\hat{A}/\mathfrak{m}^{n+1}\hat{A}\cong A_n$ (see \cite[Theorem~7.1~b)]{eisenbud1995commutative}). Now, for every natural number $n$, consider the following diagram of solid arrows
\[
\begin{tikzcd}
\mathcal{X}_n \arrow[d, dashed, "f_n"] \arrow[r, hook, dashed] & \mathcal{X} \arrow[d, "f"]\\
\Spec A_n \arrow[r, hook] & \Spec A
\end{tikzcd}
\]
and complete it to a Cartesian one. For every non-negative integer $n$, we have that $f_n\colon\mathcal{X}_n\to\Spec A_n$ is a deformation of $X$ and all these deformations satisfy the compatibility condition of \cref{eq: compatibility of deformation to induced a formal deformation}. By applying \Cref{prop: formal deformation as compatible collection of deformations} we have constructed a formal deformation $\mathfrak{f}\colon\mathfrak{X}\to\Spf\hat{A}$.
\end{remark}

We call the formal deformation $\mathfrak{f}$ constructed in \cref{remark: constructing formal deformation from classic deformation} the \emph{formal deformation associated} to $f$.

\subsection{Relations among different types of deformations}
It is now a good time to exploit the relationships among the deformations we will find in this article. We start by recalling a few definitions taken from \cite{sernesi2007deformations}.

\begin{definition}
Let $X$ be a proper scheme over an algebraically closed field $k$ and consider the following Cartesian diagram of schemes
\begin{equation}\label{eq: deformation}
\begin{tikzcd}
X \arrow[d] \arrow[r, hook] & \mathcal{X} \arrow[d, "f"]\\
\Spec k \arrow[r, hook, "b"'] & B
\end{tikzcd}
\end{equation}
with $f$ flat, proper and surjective morphism, $b\in B$ a closed point inducing the closed embedding $b\colon\Spec k\hookrightarrow B$. We say \cref{eq: deformation} is
\begin{itemize}
\item[(a)] a \textit{family of deformations} of $X$ iff $B$ is a connected $k$-scheme;
\item[(b)] an \textit{algebraic deformation} of $X$ iff $B$ is a $k$-scheme (essentially) of finite type;
\item[(c)] a \textit{local deformation} of $X$ iff $B$ is the affine spectrum of a local Noetherian $k$-algebra with residue field $k$;
\item[(d)] an \textit{infinitesimal deformation} of $X$ iff $B=\Spec A$ with $A$ a local Artinian $k$-algebra with residue field $k$;
\item[(e)] a \textit{first-order deformation} of $X$ iff $B=\Spec k[\varepsilon]/(\varepsilon^2)$.
\item[(f)] We say that a Cartesian diagram
\begin{equation*}
\begin{tikzcd}
X \arrow[r, hook] \arrow[d] & \mathfrak{X} \arrow[d, "\mathfrak{f}"]\\
\Spec k \arrow[r, hook] & \Spf A
\end{tikzcd}
\end{equation*}
of formal schemes is a \textit{formal deformation} iff $A$ is a local complete Noetherian $k$-algebra with residue field $k$ and $\mathfrak{f}$ is a flat proper morphism of finite type of formal schemes. As we have shown, this is equivalent to give a collection of infinitesimal deformations $\{f_n\colon X_n\to B_n\}_{n\in\mathbb{N}}$, where $B_n:=\Spec A/\mathfrak{m}_{A}^{n+1}$, such that the following diagram is Cartesian
\[
\begin{tikzcd}
X_n\arrow[d, "f_n"] \arrow[r, hook] & X_{n+1}\arrow[d, "f_{n+1}"]\\
B_n\arrow[r, hook] & B_{n+1}.
\end{tikzcd}
\]
\end{itemize}
\end{definition}

We remark that in cases (c), (d), (e) and (f) the underlying topological spaces of $X$ and $\mathcal{X}$ (respectively $\mathfrak{X}$) are the same and what is changing is the scheme (respectively formal scheme) structure. In particular it follows that the properness condition of $X$ is equivalent to $f$ (respectively $\mathfrak{f}$) being proper.

In the same hypotheses and notations used in the previous definition, we have the following properties:
\begin{enumerate}
\item any algebraic deformation induces a local one by taking the closed point $b\in B$ and considering the pull-back of $f\colon\mathcal{X}\to B$ along the closed embedding $\Spec\mathcal{O}_{B,b}\hookrightarrow B$; \item any infinitesimal deformation is in particular a local deformation since every Artinian ring is Noetherian too;
\item any first order deformation is an infinitesimal one because the ring of dual numbers $k[\varepsilon]/(\varepsilon^2)$ is an example of Artinian ring;
\item since, by definition, a formal deformation is a (numerable) collection of infinitesimal deformations, we get that any formal deformation induces countably many infinitesimal deformation;
\item on the other hand, any local deformation induces a formal one. To see this, let $\mathfrak{m}_{b}$ denotes the maximal ideal of the local ring $\mathcal{O}_{B,b}$ and, for any $n\in\mathbb{N}$, consider the following diagram made by Cartesian faces
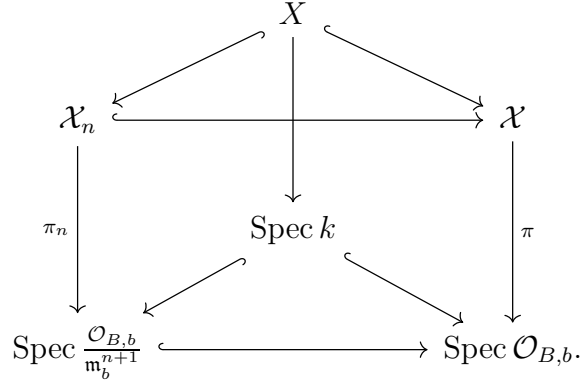
\[
\begin{tikzcd}
& X \arrow[dd] \arrow[dl, hook] \arrow[dr, hook] &\\
\mathcal{X}_n \arrow[dd, "\pi_n"'] \arrow[rr, hook] & & \mathcal{X}\arrow[dd, "\pi"]\\
& \Spec k \arrow[dl, hook] \arrow[dr, hook] &\\
\Spec\frac{\mathcal{O}_{B,b}}{\mathfrak{m}_{b}^{n+1}}\arrow[rr, hook] & & \Spec\mathcal{O}_{B,b}.
\end{tikzcd}
\]
Doing this for every $n\in\mathbb{N}$ we get a collection of compatible deformations of $X$, which defines a formal deformation.
\end{enumerate}

In this work there are more steps to be aware of. To explain them, let us consider the following diagram of deformations of a $k$-scheme $X$ (this simply means that each vertical arrow is a deformation of $X$):
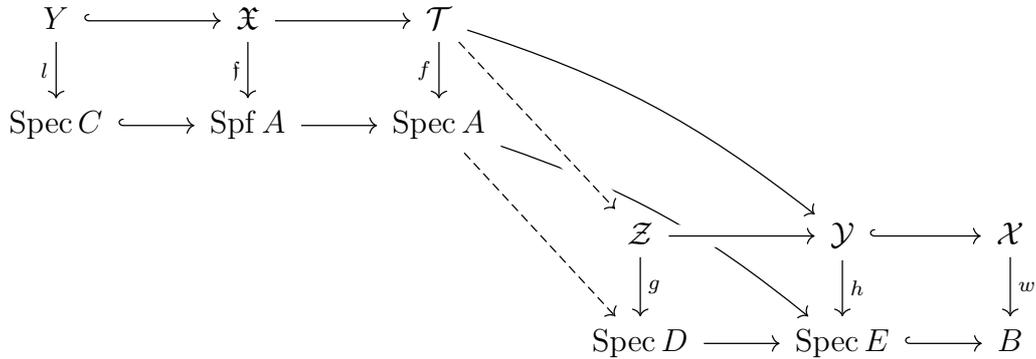
\[
\begin{tikzcd}
Y\arrow[d, "l"']\arrow[r, hook] & \mathfrak{X} \arrow[d, "\mathfrak{f}"']\arrow[r] &\mathcal{T} \arrow[d, "f"'] \arrow[rrdd, bend left = 10] &&&\\
\Spec C \arrow[r, hook] & \Spf A \arrow[r] & \Spec A \arrow[rdd, dashed] \arrow[rrdd, bend left = 10] &&&\\
&&& \arrow[from=luu, crossing over, dashed] \mathcal{Z} \arrow[d, "g"] \arrow[r, crossing over] &\mathcal{Y} \arrow[d, "h"] \arrow[r, hook] & \mathcal{X} \arrow[d, "w"]\\
&&& \Spec D \arrow[r] &\Spec E \arrow[r, hook] & B\\
\end{tikzcd}
\]
where $B$ is a $k$-scheme of finite type, $E$ is a $k$-algebra (essentially) of finite type (essentially of finite type means that it is the localization of a $k$-algebra of finite type), $D$ is a $k$-algebra which is also a DVR, $A$ is a local complete Noetherian $k$-algebra and $C$ is a local Artinian $k$-algebra.

We will say that a morphism defining a deformation is induced by another if the second deformation is isomorphic (as deformations, see \cite[page $21$]{sernesi2007deformations}) to the pull-back of the first along the closed embedding on the base. We point out that, in general, there is not a natural arrow from $f$ to $g$, hence the dashed arrow, unless $A$ is taken to be the completion of the DVR $D$ along its maximal ideal. Now, $h$ is induced by $w$ since the closed embedding $b\colon\Spec k\to B$ factors through the spectrum of a $k$-algebra (essentially) of finite type. Passing from $h$ to $f$ can be done as follows: since $E$ is the localization of a $k$-algebra of finite type, it has a maximal ideal $\mathfrak{m}_{E}$ and we can complete $E$ along such maximal ideal. Similarly, from $g$ we can deduce $f$ by considering the completion of the DVR along the powers of its maximal ideal. $f$ induces $\mathfrak{f}$ since the formal spectrum has a natural map to the affine spectrum. Since the quotient of a DVR by powers of its maximal ideal is an Artinian ring, it follows that $g$ induces $l$. Similarly, $f$ induces $l$. Lastly, the formal deformation $\mathfrak{f}$ induces a infinitesimal deformation $l$ since the quotient of a local complete Noetherian ring by a power of the maximal ideal is an Artinian ring.

\subsection{Reversing some constructions on deformation}
Reversing some constructions above is usually a hard problem and without further hypotheses on the scheme $X$ is a very hard one. For example, passing from a formal deformation of a $k$-scheme $X$ to a deformation of the same scheme over an affine spectrum of a $k$-algebra (essentially) of finite type means to find ``an algebraisation of the formal deformation''. By an algebraisable formal deformation we mean the following:

\begin{definition}\label{def: algebraic deformation}
Let $X$ be a scheme and let $(A,\mathfrak{m})$ be a complete local Noetherian ring. A formal deformation $\mathfrak{f}\colon\mathfrak{X}\to\Spf A$ is called \emph{algebraisable} if there exist
\begin{itemize}
\item a $k$-algebra essentially of finite type $(R, \mathfrak{n})$,
\item a deformation $g\colon\mathcal{Y}\to\Spec R$ of $X$,
\item an isomorphism $A\cong\hat{R}_{\mathfrak{n}}$,
\item an isomorphism between $\mathfrak{f}$ and the formal deformation $\mathfrak{g}\colon\mathfrak{Y}\to\Spf\hat{R}$ associated to $g$.
\end{itemize}
The deformation $g\colon\mathcal{Y}\to\Spec A$ is called an \emph{algebraisation} of $\mathfrak{f}$.
\end{definition}

The existence of an algebraisation is a very difficult problem. To solve it, Artin introduced in \cite{artin1969algebraization} a weaker condition than algebraisation, ``effectivity of a formal deformation'', which we introduce next.

\begin{definition}\label{def: effective deformation}
Let $X$ be a scheme and let $(A,\mathfrak{m})$ be a complete local Noetherian ring. A formal deformation $\mathfrak{f}\colon\mathfrak{X}\to\Spf A$ is called \emph{effective} if there exists a deformation
\[
\begin{tikzcd}
X \arrow[r, hook] \arrow[d] & \mathcal{X} \arrow[d, "f"]\\
\Spec k \arrow[r, hook] & \Spec A
\end{tikzcd}
\]
with $f$ a flat morphism of finite type such that $\mathfrak{X}=\hat{\mathcal{X}}_{/X}$.
\end{definition}

The idea of Artin was to split the problem of algebraisation in two subproblems:
\begin{itemize}
\item[(i)] to prove the effectivity of the formal deformation: in other words, using notations above, find conditions on $X$ to extend the formal deformation $\mathfrak{f}$ to the deformation $f$;
\item[(ii)] to find hypotheses on the formal deformation to extend it to the spectrum of an (essentially) of finite type $k$-algebra.
\end{itemize}

For step (ii), a sufficient criterion was given by Artin in \cite{artin1969algebraization} and goes by the name of Artin algebraisation theorem, see \cite[Theorem 2.5.14]{sernesi2007deformations}. In there, under the hypotheses that the central fibre $X$ is a projective scheme, Artin showed that if the formal deformation is versal, see \cite[Definition~2.2.6]{sernesi2007deformations}, and effective then it is algebraisable.

However, step (i) above can not be always achieved: for instance, the universal formal deformation of a $K3$ surface is not effective, see \cite[Example 2.5.12]{sernesi2007deformations}.

Recall from \Cref{def: algebraisable formal scheme} that a LNFS $\mathfrak{X}$ is called algebraizable if there are a scheme $Y$ and a closed subscheme $X$ of $Y$ such that $\mathfrak{X}=\hat{Y}_{/X}$. It also makes sense to define algebraisable schemes in the relative setting. For this, suppose we have a formal scheme $\mathfrak{X}$ over the affine formal scheme $\Spf A$, with $A$ a local adic Noetherian $k$-algebra $A$ with residue field $k$.

\begin{definition}
We say that $\mathfrak{X}$ is \emph{algebraisable} over $\Spf A$ if there exists a scheme $\mathcal{X}$ over $\Spec A$ such that $\mathfrak{X}$ is isomorphic to the formal completion $\widehat{\mathcal{X}}_{/X_{0}}$, where $X_{0}:=\mathfrak{X}\times_{\Spf A}\Spec\frac{A}{I}$.

When the formal scheme $\mathfrak{X}$ is proper over $\Spf A$, we say that it is algebraisable if there exists a proper scheme $\mathcal{X}$ over $\Spec A$ such that $\mathfrak{X}\cong\widehat{\mathcal{X}}_{/X_0}$, where again $X_{0}:=\mathfrak{X}\times_{\Spf A}\Spec\frac{A}{I}$.
\end{definition}

Note that the ring $A$ is left fixed but we are changing the locally ringed space structure induced by it.

At this point one wonders if there are conditions to ensure algebraisability of a formal scheme and, in the case an algebraisation exists, how unique it is. We first address the latter problem. Assume we have found an algebraisation of a locally Noetherian formal scheme; then in general it is not unique and a counterexample is given in \cite[Remark~8.4.8.]{FGAIll}. However, in \cite[Corollary~8.4.7.]{FGAIll}, Illusie proved that if we restrict to proper formal schemes then an algebraisation is unique up to a unique isomorphism inducing the identity on $\mathfrak{X}$.

\begin{theorem}[{\cite[(\textbf{3}.5.4.5)]{MR0163910} or \cite[Theorem~8.4.10]{FGAIll}}]\label{Gro: algebrization theorem}
Let $A$ be a Noetherian $I$-adic ring, let $T=\Spec A$, $\mathfrak{T}:=\Spf A$, let $\mathfrak{f}\colon\mathfrak{X}\to\mathfrak{T}$ be a proper morphism of formal schemes. For any $l\in\mathbb{N}$, let $T_l:=\Spec (A/I^{l+1})$, $X_l:=\mathfrak{X}\times_{\mathfrak{T}}T_l$. Suppose that there is an invertible formal sheaf $\mathfrak{L}$ such that $\mathfrak{L}_0:=\mathfrak{L}/I\mathfrak{L}$ is an ample invertible sheaf on $X_0$. Then $\mathfrak{X}$ is algebraisable. 
Furthermore, if $\mathcal{X}$ is its algebraisation, which is proper over $T$ since $\mathfrak{f}$ was supposed proper, then there exists a unique ample invertible sheaf $\mathcal{M}$ on $\mathcal{X}$ such that $\mathfrak{L}\cong\widehat{\mathcal{M}}_{/X_0}$.
\end{theorem}

We present now a few remarks on the above theorem. The first one is that the hypotheses $\mathfrak{f}$ proper, which is equivalent to $X_0$ proper over $k$, and $\mathfrak{L}_{0}$ ample on $X_{0}$ together imply that $X_0$ is projective over $T_0$. Furthermore, since $\mathfrak{f}$ was proper, the algebraisation of $\mathfrak{X}$ is proper over $T$ by the above definition; in particular it is unique up to a unique isomorphism inducing the identinty on the formal scheme. We also remark that the existence of an ample invertible sheaf $\mathcal{M}$ on $\mathcal{X}$ together with the fact that $\mathcal{X}$ is proper, imply that $\mathcal{X}$ is projective over $T$.

A corollary of the above theorem, is the classical result by Grothendieck:

\begin{theorem}[{\cite[Th{\'e}or{\`e}me~4]{GAGF1960}}]\label{theorem: effectivisation of deformation II}
Let $A$ be a local adic Noetherian ring with residue field $k$, let $\mathfrak{X}$ be a proper formal scheme over $\Spf A$ and suppose that
\begin{enumerate}
\item the local rings of $\mathcal{O}_{\mathfrak{X}}$ are flat $A$-modules (in other words $\mathfrak{f}$ is flat);
\item $X_0:=\mathfrak{X}\otimes_{A}k$ satisfies $\coho^2(X_0,\mathcal{O}_{X_0})=0$;
\item $X_0$ is projective.
\end{enumerate}
Then $\mathfrak{X}$ is algebraisable and its algebraisation is projective over $\Spec A$.
\end{theorem}

We can interpret \Cref{theorem: effectivisation of deformation II} as a theorem on deformations; it says that, in the same notations as above, if the structure morphism $\mathfrak{f}\colon\mathfrak{X}\to\Spf A$ is proper and a formal deformation of a projective scheme $X_0$ with $\coho^2(X_0, \mathcal{O}_{X_0})=0$, then the formal deformation $\mathfrak{f}$ is effective. Therefore \Cref{theorem: effectivisation of deformation II} gives sufficient conditions to achieve step (i) above.

The difference between the algebraisation of a formal scheme over a formal affine scheme, say $\Spf A$ with $A$ as above, and the algebraisation of a formal deformation over $\Spf A$ with proper central fibre lies in the base affine scheme: in the algebraisation of the formal scheme, the $k$-algebra is required to be complete, while in the algebraisation of the formal deformation the $k$-algebra is required to be (essentially) of finite type.

Examples of formal schemes that are not algebraizable (resp. formal deformations that are not effective) are $K3$ surfaces and Abelian varieties, see \cite[Example~2.5.12 ]{sernesi2007deformations} or \cite[Remark~8.5.24(b) and remark~8.5.28(a)]{FGAIll}. Even though in both cases we are able to extend (resp. deform) the scheme at all infinitesimal neighbourhoods, there are ample line bundles that do not lift to the whole formal scheme. This is the consequence of the fact that the space of all deformation of pairs Abelian variety together with an ample line bundle on it (or $K3$ surface together with an ample line bundle) is a proper subspace of the space of all deformations of Abelian varieties (or of $K3$ surfaces).

\subsection{Motivating formal smoothness}
In the next part we introduce two definitions of smoothing that will be relevant in the following. In particular we motivate why the definition of formal smoothing given by Tziolas in \cite{tziolas2010smoothings} is the most natural and, in some sense, the only one possible in our framework.

This section was motivated by the following result of Tziolas, which is key to our argument.

\begin{proposition}[{\cite[Proposition~11.8]{tziolas2010smoothings}}]\label{Tziolas: smoothing over dvr is the same as formal smoothing}
Let $Y$ be a proper, equidimensional scheme and let $A$ be a $k$-algebra which is a DVR. Let $g\colon\mathcal{Y}\to\Spec A$ also be a deformation of $Y$ over $A$ and let $\mathfrak{g}\colon\mathfrak{Y}\to\Spf\hat{A}$ be the associated formal deformation. Then $g$ is a smoothing if and only if $\mathfrak{g}$ is a formal smoothing.
\end{proposition}

The importance of the above result is that it gives a criterion to recognise if a one-parameter deformation is a smoothing by checking if the associated formal deformation is a formal smoothing. Let us first introduce the two definitions of smoothings.

\begin{definition}\label{def: smoothing}
Let $Y$ be a proper, equidimensional scheme and let $A$ be a $k$-algebra which is a DVR. We say that a deformation $g\colon\mathcal{Y}\to\Spec A$ of $Y$ over $A$ is a \emph{smoothing} if the generic fibre $\mathcal{Y}_{\text{gen}}:=\mathcal{Y}\times_{\Spec A}\Spec\kappa(A)$ is smooth.
\end{definition}

Following \cite{tziolas2010smoothings}, we now recall the notion of formal smoothing. Such definition requires the knowledge of the sheaf of Fitting ideals, which can be found either in \cite[Chapter~20.2]{eisenbud1995commutative} or in \cite[\href{https://stacks.math.columbia.edu/tag/0C3C}{TAG 0C3C}]{stacks-project}. We will not introduce it but we will just give an interpretation of what the Fitting ideal is. Let $\mathfrak{X}$ be a formal scheme, let $\mathfrak{F}$ be a formal coherent sheaf and let $a\in\mathbb{N}$; we denote by $\Fitt_a(\mathfrak{F})$ the $a^{\text{th}}$ Fitting ideal sheaf of $\mathfrak{F}$. This ideal measures the obstructions for the sheaf $\mathfrak{F}$ to be locally generated by $a$ elements. For example, $\mathfrak{F}$ is locally generated by $a$ elements if and only if $\Fitt_a(\mathfrak{F})=\mathcal{O}_{\mathfrak{X}}$.

\begin{definition}\label{def: formal smoothing}
Let $X$ be a proper, equidimensional scheme. A formal deformation of $X$ over $\mathfrak{S}$
\[
\begin{tikzcd}
X \arrow[d] \arrow[r, hook] & \mathfrak{X} \arrow[d, "\pi"]\\
\Spf k \arrow[r, hook] & \mathfrak{S}
\end{tikzcd}
\]
is called a \emph{formal smoothing} of $X$ if and only if there exists a natural number $a$ such that $\mathfrak{I}^a\subset\Fitt_{\dim X}(\Omega^1_{\mathfrak{X}/\mathfrak{S}})$, where $\mathfrak{I}$ is an ideal of definition of $\mathfrak{X}$ and $\Fitt_{\dim X}(\Omega^1_{\mathfrak{X}/\mathfrak{S}})$ is the Fitting sheaf of ideals.

We say that $X$ is \emph{formally smoothable} if it admits a formal smoothing.
\end{definition}

We point out that \Cref{Tziolas: smoothing over dvr is the same as formal smoothing} establishes an equivalence among two different notions of smoothing that, apparently, are very different. Indeed, in \Cref{def: smoothing}, the condition uses strongly the existence of a generic point, while in \Cref{def: formal smoothing} the same condition is ``forced'' to be algebraic since there is not a generic point in $\Spf k\llbracket t\rrbracket$.

As mentioned above the two notions differ only apparently as we are going to explain next.

First observe that any DVR which is a $k$-algebra is a local Noetherian ring; in particular we have that its completion with respect to the adic topology induced by its maximal ideal is isomorphic to the formal power series in one variable, $k\llbracket t\rrbracket$. We also remark that the (classical) spectrum of a DVR contains two points: the closed and the generic one. On the other hand, the formal spectrum of the formal power series ring is made of one point only. Therefore it is natural to define the notion of smoothing of a scheme over a DVR as a deformation of $X$ whose general fibre, i.e. the fibre over the open generic point, is smooth. On the other hand, in the case of formal deformation over $\Spf k\llbracket t \rrbracket$ such idea is not possible. However, Tziolas come up with a definition of formal smoothing that does not need the generic point, as we are going to explain now. Let us suppose that $\pi\colon\mathcal{X}\to B$ is a locally of finite type, flat of relative dimension $r$ morphism of schemes and define
\[
U_{r}=\left\{x\in\mathcal{X}\colon \pi\text{ is smooth at $x$ of relative dimension }r\right\}.
\]
By \cite[\href{https://stacks.math.columbia.edu/tag/02G2}{TAG 02G2}]{stacks-project}, it is an open subset of $\mathcal{X}$ and by \cite[407]{eisenbud1995commutative} or \cite[\href{https://stacks.math.columbia.edu/tag/0C3K}{TAG 0C3K}]{stacks-project} we have that
\[
U_{r}=\mathcal{X}\setminus\text{V}(\Fitt_{r}(\Omega^1_{\pi}))\,\,\,\,\,\,\,\,\text{ and }\,\,\,\,\,\,\,\,\text{Sing}_{r}(\pi)=\text{V}(\Fitt_r(\Omega^1_{\pi})).
\]
If we assume that $\pi$ is proper, then $\pi(U_r)\subset B$ is open too and $\pi|_{U_{r}}\colon U_r\to A_{r}$ is smooth of relative dimension $r$, where $A_r:=B\setminus\pi(\text{V}(\Fitt_r(\Omega^1_{\pi})))$. Doing a base change, we can always find a smoothing from the family over $B$ if and only if $A_r$ is not empty.

Suppose that $B$ is affine smooth curve over $k$, let $p\in B$ be a closed point and let $R:=\mathcal{O}_{B,p}$; it is known that $R$ is a DVR with residue field $k$. $\pi\colon\mathcal{X}\to B$ is a smoothing (according to \Cref{def: smoothing}) if and only if the pullback deformation $\mathcal{X}_{R}\to\Spec R$ along the localization morphism $\Spec R\to B$ is a smoothing (again in the sense of \Cref{def: smoothing}). We then have following diagram:
\[
\begin{tikzcd}
\mathcal{X}_{n} \arrow[d, "\pi|_{\mathcal{X}_n}"'] \arrow[r, hook] & \mathcal{X}_{\widehat{R}} \arrow[d, "\pi|_{\mathcal{X}_{\widehat{R}}}"'] \arrow[r, hook, "\beta"] & \mathcal{X}_R \arrow[d, "\pi|_{\mathcal{X}_R}"] \arrow[r, hook] & \mathcal{X} \arrow[d, "\pi"]\\
S_n \arrow[r, hook] & \Spec \widehat{R} \arrow[r, hook, "\alpha"'] & \Spec R \arrow[r, hook] & B
\end{tikzcd}
\]
where all squares are Cartesian, $\widehat{R}$ denotes the completion of $\mathcal{O}_{B,p}$ along its maximal ideal $\mathfrak{m}_{p}$ and, for every $n\in\mathbb{N}$, $R_n:=\frac{k\llbracket t\rrbracket}{(t^{n+1})}=\frac{k[t]}{(t^{n+1})}$ and $S_n:=\Spec R_n$. As previously mentioned, the completion of $\mathcal{O}_{B,p}$ along the maximal ideal is isomorphic to $k\llbracket t \rrbracket$.

In order to lighten the notation, let us denote $\pi_n:=\pi|_{\mathcal{X}_n}$, $\widehat{\pi}:=\pi|_{\mathcal{X}_{\widehat{R}}}$, $\widetilde{\pi}:=\pi|_{\mathcal{X}_R}$.

Observe now that $\alpha$ is a homeomorphism, hence $\beta$ is at least a bijective function on the sets; by \cite[Corollary~20.5]{eisenbud1995commutative} we have that
\[
\text{Sing}_r(\widehat{\pi})=\beta^{-1}(\text{Sing}_r(\widetilde{\pi})).
\]
Therefore, $\widetilde{\pi}$ is smooth of relative dimension $r$ along $\widetilde{\pi}^{-1}(\eta)$ if and only if $\widehat{\pi}$ is smooth of relative dimension $r$ along $\widehat{\pi}^{-1}(\widehat{\eta})$, where $\eta$ and $\widehat{\eta}$ are the generic points of $\Spec R$ and $\Spec\widehat{R}$ respectively. Now $\Spec R$ has only two points: the closed one, $Y$ with ideal sheaf $\mathcal{I}_{Y/\Spec R}=(t)$, and the open one, $\eta$. Let $C_r:=\text{Sing}_{r}(\widetilde{\pi})$. Now $\widetilde{\pi}(C_{r})\subset Y$ as schemes if and only if there exists a structure of closed $Spec R$-subscheme $\widetilde{Y}$ on $Y$ with $\widetilde{Y}_{\text{red}}=Y$ and such that $\widetilde{\pi}(C_{r})\subset\widetilde{Y}$ as sets. We are then reduced to classify all closed subscheme structures on $\Spec k\llbracket t\rrbracket$. These are given by $Y_{k}:=\text{V}((t^{k+1}))$, for every $k\in\mathbb{N}$. In particular we have a chain of closed subschemes
\[
Y=\text{V}((t))=Y_{0}\subset Y_1=\text{V}((t^2))\subset Y_3\subset\cdots.
\]
Hence, $\widetilde{Y}$ is a closed $\Spec R$-subscheme structure on $Y$ satisfying $\widetilde{Y}_{\text{red}}=Y$ and $\widetilde{\pi}(C_{r})\subset\widetilde{Y}$ if and only if there exists a non-negative integer $k$ such that $\widetilde{Y}=Y_{k}$.
Concluding, we have proven that the following statements are equivalent:
\begin{itemize}
\item[(a)] $\pi\colon\mathcal{X}\to B$ is smooth of relative dimension $r$;
\item[(b)] $\widetilde{\pi}\colon\mathcal{X}_{R}\to\Spec R$ is smooth of relative dimension $r$;
\item[(c)] $\widehat{\pi}\colon\mathcal{X}_{\widehat{R}}\to\Spec\widehat{R}$ is smooth of relative dimension $r$;
\item[(d)] there is a closed subscheme $\widetilde{Y}$ of $\Spec R$ such that $\widetilde{Y}_{\text{red}}=Y$ and $\widetilde{\pi}(C_{r})\subset\widetilde{Y}$;
\item[(e)] there exist a $k\in\mathbb{N}$ such that $\widetilde{\pi}(C_{r})\subset Y_{k}$;
\item[(f)] there exists a $k\in\mathbb{N}$ such that $C_{r}\subset\widetilde{\pi}^{-1}(Y_{k})$;
\item[(g)] there exists a $k\in\mathbb{N}$ such that
\[
\Fitt_r(\Omega^1_{\widetilde{\pi}})=\mathcal{I}_{C_{r}/\mathcal{X}}\supseteq\widetilde{\pi}^{-1}((t^{k+1}))=\widetilde{\pi}^{-1}(\mathcal{I}_{Y_{k}/\Spec R}).
\]
\end{itemize}
Observing that the condition we have found is independent of the ideal of definition, we have reached the definition of formal smoothing as given in \cite[Definition~11.6]{tziolas2010smoothings}.

\section{Gorenstein schemes, morphisms and their deformations}
In this part we will review, following \cite[\href{https://stacks.math.columbia.edu/tag/08XG}{Tag 08XG}]{stacks-project} and \cite[\href{https://stacks.math.columbia.edu/tag/0DWE}{Tag 0WDE}]{stacks-project}, the notions of dualising complexe and of Gorenstein morphisms. We then discuss how the Gorenstein property behaves under infinitesimal deformations. The main result of this section is that the relative dualising sheaf extends to every infinitesimal deformation. In the way to prove this result, we also present a proof of the classical result that deformation of a Gorenstein morphism is still Gorenstein, for which we were not able to find a proof in the literature.

\subsection{Gorenstein schemes and morphisms}
We start the section introducing the notions of dualising sheaf, Gorenstein scheme and Gorenstein morphism.

\begin{definition}\label{def: dualising complex}
Let $A$ be a Noetherian ring. A \emph{dualising complex} is a complex of $A$ modules $\omega_A^{\bullet}$ such that
\begin{enumerate}
\item $\omega_A^{\bullet}$ has finite injective dimension;
\item $\coho^i(\omega_A^{\bullet})$ is a finite $A$-module, for every $i$;
\item $A\to\mathbf{R}\Hom_A(\omega_A^{\bullet},\omega_A^{\bullet})$ is a quasi-isomorphism in the derived category of $A$-modules.
\end{enumerate}
\end{definition}

We remark that the dualising complex thus defined is not unique. Indeed, according to \cite[\href{https://stacks.math.columbia.edu/tag/0A7F}{TAG 0A7F}]{stacks-project}, if $\omega_{A}^{\bullet}$ and $\nu_{A}^{\bullet}$ are two dualising complexes for $A$, then there exists an invertible object $L^{\bullet}\in\text{D}(A)$ such that $\nu_{A}^{\bullet}$ is quasi-isomorphic to $\omega_{A}^{\bullet}\otimes_{A}^{\textbf{L}}L^{\bullet}$.

\begin{definition}\label{def: Gorenstein local ring}
Let $A$ be a local Noetherian ring. We say that $A$ is a \emph{Gorenstein local ring} if $A[0]$ is a dualising complex.
\end{definition}


\begin{definition}\label{def: Gorenstein scheme}
A scheme $X$ is called \emph{Gorenstein} if it is locally Noetherian and for every $x\in X$, $\mathcal{O}_{x, X}$ is a Gorenstein local ring according to \Cref{def: Gorenstein local ring}.
\end{definition}

\begin{definition}\label{def: Gorenstein morphism}
Let $f\colon X\to Y$ be a morphism of schemes such that for every $y\in Y$, the fibre $X_y$ is a locally Noetherian scheme.
\begin{enumerate}
\item Let $x\in X$ and $y:=f(x)$. We say that $f$ is \emph{Gorenstein at $x$} if $f$ is flat at $x$ and $\mathcal{O}_{X_y,x}$ is a Gorenstein local ring.
\item We say that $f$ is \emph{Gorenstein} if it is Gorenstein at $x$, for all $x\in X$.
\end{enumerate}
\end{definition}

\begin{lemma}[{\cite[\href{https://stacks.math.columbia.edu/tag/0C12}{Tag 0C12}]{stacks-project}}]\label{lemma: Gorenstein scheme and flat morphism imply Gorenstein morphism}
Let $f\colon X\to Y$ be a flat morphism of locally Noetherian schemes. If $X$ is Gorenstein, then $f$ is Gorenstein.
\end{lemma}

\begin{proposition}[{\cite[\href{https://stacks.math.columbia.edu/tag/0C07}{Tag 0C07}]{stacks-project}}]\label{prop: base change of Gorenstein morphism is Gorenstein}
Let $f\colon X\to Y$ be a morphism of schemes such that for every $y\in Y$ the fiber $X_y$ is locally Noetherian and let $g\colon Y'\to Y$ be a locally of finite type morphism of schemes. Consider the following Cartesian diagram
\begin{equation}\label{eq: fibre product diagram}
\begin{tikzcd}
X' \arrow[d, "f'"] \arrow[r, "g'"] & X \arrow[d, "f"]\\
Y' \arrow[r, "g"] & Y.
\end{tikzcd}
\end{equation}
If $f'$ is Gorenstein at $x'\in X'$ and $f$ is flat at $g'(x')$, then $f$ is Gorenstein at $g'(x')$.
\end{proposition}

From this it follows that being a Gorenstein is local in the flat topology on the category of schemes.

\subsection{Right adjoint to the pushforward and relative dualising complex}
Now we introduce the derived pushforward functor and its right adjoint. This machinery will be used to define a relative dualising complex and to show that it behaves well under pullbacks.

\begin{definition}
Let $f\colon X\to Y$ be a morphism of scheme with $Y$ quasi-compact. By \cite[\href{https://stacks.math.columbia.edu/tag/0A9E}{Tag 0A9E}]{stacks-project}, $\textbf{R}f_*\colon\text{D}_{\text{QCoh}}(X)\to\text{D}_{\text{QCoh}}(Y)$ admits a right adjoint and we denote it by $\Psi\colon\text{D}_{\text{QCoh}}(Y)\to\text{D}_{\text{QCoh}}(X)$.
\end{definition}

\begin{definition}\label{def: relative dualising complex}
Let $Y$ be a quasi-compact scheme, let $f\colon X\to Y$ be a proper, flat morphism of finite presentation and let $\Psi$ be the right adjoint for $\mathbf{R}f_*$. We define the \emph{relative dualising complex} $\omega_f^{\bullet}$ of $f$ (or of $X$ over $Y$) as follows
\[
\omega_f^{\bullet}:=\Psi(\mathcal{O}_Y)\in\text{D}_{\text{QCoh}}(X).
\]
\end{definition}

The following proposition explains the behaviour of the relative dualising complex under base change.

\begin{proposition}[{\cite[\href{https://stacks.math.columbia.edu/tag/0AAB}{Tag 0AAB}]{stacks-project}}]\label{coro: base change for dualising complex}
Let $X$ be a scheme, let $Y$ and $Y'$ be quasi-compact schemes, let $g\colon Y'\to Y$ also be any morphism and let $f\colon X\to Y$ be a proper, flat morphism of finite presentation. Consider the fibre diagram as in \cref{eq: fibre product diagram}. Then we have a canonical isomorphism
\[
\omega_{f'}^{\bullet}\cong\mathbf{L}(g')^*\omega_f^{\bullet}\in\text{D}_{\text{QCoh}}(X'),
\]
where $X':=X\times_{Y}Y'$.
\end{proposition}

\subsection{Upper shriek functor and Gorenstein morphisms}
We now introduce the upper shriek functor and explain its relationships with the right adjoint functor for the derived pushforward functor and with Gorenstein morphisms.

Remember from \Cref{conventions} that FTS is the category whose objects are separated, algebraic schemes over the field $k$ and whose morphisms are morphisms of $k$-schemes.

\begin{definition}\label{def: upper shriek functor}
Let $f\colon X\to Y$ be a morphism in the category of FTS schemes. We define the upper shriek functor
\[
f^!\colon\text{D}^+_{\text{QCoh}}(\mathcal{O}_Y)\to\text{D}^+_{\text{QCoh}}(\mathcal{O}_X)
\]
as follows. We choose a compactification $X\to\bar{X}$ of $X$ over $Y$. Such a compactification always exists by \cite[\href{https://stacks.math.columbia.edu/tag/0F41}{Tag 0F41}]{stacks-project} and \cite[\href{https://stacks.math.columbia.edu/tag/0A9Z}{Tag 0A9Z}]{stacks-project}. Let denote by $\bar{f}\colon\bar{X}\to Y$ the structure morphism and consider its right adjoint functor $\bar{\Psi}$; we then let $f^! K:=\bar{\Psi}(K)|_{X}$ for $K\in\text{D}^+_{\text{QCoh}}(\mathcal{O}_Y)$.
\end{definition}

According to \cite[\href{https://stacks.math.columbia.edu/tag/0AA0}{Tag 0AA0}]{stacks-project}, the definition of the upper shriek functor is, up to canonical isomorphism, independent of the choice of the compactification of $X$.

\begin{remark}\label{rem: relationship between upper shriek and right derived functor in case of proper morphism}
We point out that if $f\colon X\to Y$ is a proper morphism in the category FTS, then $\bar{\Psi}=\Psi$, implying that the upper shriek functor is the restriction to $\text{D}_{\text{QCoh}}(\mathcal{O}_Y)$ of $\Psi$, the right adjoint functor of $\mathbf{R}f_*$ (see \cite[\href{https://stacks.math.columbia.edu/tag/0AU3}{Tag 0AU3}]{stacks-project}).
\end{remark}

We are now ready to present the link between the Gorenstein condition and the upper shriek functor.

\begin{proposition}[{\cite[\href{https://stacks.math.columbia.edu/tag/0C08}{Tag 0C08}]{stacks-project}}]\label{prop: Gorenstein iff dualising complex is invertible in an open subset}
Consider $f\colon X\to Y$ a flat morphism of schemes in FTS and let $x\in X$. Then the following conditions are equivalent:
\begin{enumerate}
\item $f$ is Gorenstein at $x$;
\item $f^!\mathcal{O}_Y$ is isomorphic to an invertible object (of the derived category) in a neighbourhood of $x$.
\end{enumerate}
In particular the set $\{x\in X\colon f\text{ is Gorenstein at }x\}$ is open in $X$.
\end{proposition}

If we assumed that $f$ were proper, then $\{y\in Y\colon f\text{ is Gorenstein at }x\in f^{-1}(y)\}$ is open in the target.

\subsection{Relative dualising sheaf and dualising complex}
The aim of this subsection is to show that all the definitions given until now, under mild hypotheses, converge. In particular, the next proposition introduces the notion of relative dualising sheaf for a morphism in the category FTS and describes its relationships with the relative dualising complex and with the Gorenstein morphisms.

\begin{proposition}[{\cite[\href{https://stacks.math.columbia.edu/tag/0BV8}{Tag 0BV8}]{stacks-project}}]\label{prop: existence of the relative dualising sheaf}
Let $X$ and $Y$ be separated schemes and let $f\colon X\to Y$ be a Gorenstein morphism of schemes. Then there exists a coherent, invertible sheaf, called the relative dualising sheaf of $f$ and denoted by $\omega_f$, which is flat over $Y$ and satisfies
\[
f^{!}\mathcal{O}_{Y}\cong\omega_f[-d],
\]
where $d$ is the locally constant function on $X$ which gives the relative dimension of $X$ over $Y$.

If $f$ is also proper, flat and of finite presentation, then $\omega_f^{\bullet}=\omega_f[-d]$.
\end{proposition}

If $Y=\Spec k$, then we denote the relative dualising sheaf of $X$ over $k$ by $\omega_{X}$.

\begin{proposition}\label{prop: infinitesimal deformation of Gorenstein scheme is Gorenstein}
Let $X$ be a Gorenstein scheme and let $A$ be an Artinian local $k$-algebra with residue field $k$. Consider now a deformation of $X$ over $A$; that is a Cartesian diagram
\[
\begin{tikzcd}
X \arrow[d] \arrow[r, hook] & \mathcal{X} \arrow[d, "f"]\\
\Spec k \arrow[r, hook] & \Spec A
\end{tikzcd}
\]
with $f$ flat (see \cite{sernesi2007deformations}).
Then $f$ is a Gorenstein morphism.
\end{proposition}

\begin{proof}
Since $X$ is Gorenstein and $X\to\Spec k$ is flat, by \Cref{lemma: Gorenstein scheme and flat morphism imply Gorenstein morphism} it follows that $X\to\Spec k$ is Gorenstein. Applying now \Cref{prop: base change of Gorenstein morphism is Gorenstein}, we deduce that $f\colon\mathcal{X}\to\Spec A$ is Gorenstein.
\end{proof}

\begin{remark}
The result can be improved to obtain that the scheme $\mathcal{X}$ is Gorenstein. This is true as soon as we require that the affine base scheme $\Spec A$ is the spectrum of a local, Artinian, Gorenstein $k$-algebra $A$. This result and its proof can be found in \cite{nobile2021}.
\end{remark}

Now we present the first result that will help us to deduce the existence of a geometric smoothing.

\begin{proposition}\label{dualising sheaf of a Gorenstein proper scheme always extends}
Let $X$ be a proper, Gorenstein scheme. If $\mathfrak{f}\colon\mathfrak{X}\to\mathfrak{S}$ is a formal deformation of $X$, then there exists a unique invertible formal sheaf $\mathfrak{L}$ on $\mathfrak{X}$ such that $\mathfrak{L}\otimes_{\mathcal{O}_{\mathfrak{X}}}\mathcal{O}_{X}\cong\omega_X$ and $\mathfrak{L}\otimes_{\mathcal{O}_{\mathfrak{X}}}\mathcal{O}_{\mathcal{X}_{n}}\cong\omega_{f_n}$, for every $n\in\mathbb{N}$, where $\omega_{f_n}$ is the relative dualising sheaf. In particular, every morphism $f_{n}$ is Gorenstein.
\end{proposition}

\begin{proof}
By \Cref{prop: formal deformation as compatible collection of deformations}, the formal deformation $\mathfrak{f}$ is equivalent to a collection of deformations $\left\{f_n\colon\mathcal{X}_n\to S_n\right\}_{n\in\mathbb{N}}$ satisfying the compatibility condition of \Cref{eq: compatibility of deformation to induced a formal deformation}, with $f_n$ flat, proper morphisms. Since $X$ is Gorenstein, applying \Cref{prop: infinitesimal deformation of Gorenstein scheme is Gorenstein} we deduce that for every natural number $n$, the morphism $f_n$ is Gorenstein. Now consider the following Cartesian diagram
\[
\begin{tikzcd}
\mathcal{X}_n \arrow[d, "f_n"] \arrow[r, hook, "j_{n}"] & \mathcal{X}_{n+1} \arrow[d, "f_{n+1}"]\\
S_n \arrow[r, hook] & S_{n+1};
\end{tikzcd}
\]
we have, for every natural number $n$, the following chain of equalities and natural isomorphisms
\[
\begin{aligned}
j_{n}^*\omega_{f_{n+1}}&=\coho^{-\dim X}(j_n^*\omega_{f_{n+1}}[-\dim X])\,\,\,\,\,(\text{\Cref{prop: existence of the relative dualising sheaf}})\\
&=\coho^{-\dim X}(\mathbf{L}j_n^*\omega_{f_{n+1}}^{\bullet})\\
&\cong\coho^{-\dim X}(\omega_{f_{n}}^{\bullet})\,\,\,\,\,\,\,\,\,\,\,\,\,\,\,\,\,\,\,\,\,\,\,\,\,\,\,\,\,\,\,\,\,\,\,\,\,\,\,\,(\text{\Cref{coro: base change for dualising complex}})\\
&=\coho^{-\dim X}(\omega_{f_n}[-\dim X])\,\,\,\,\,\,\,\,\,\,\,\,\,\,\,(\text{\Cref{prop: existence of the relative dualising sheaf}})\\
&=\omega_{f_n}.
\end{aligned}
\]
\Cref{teo: locally free sheaves on each nilpotent subscheme induce a locally free sheaf on the completion} then implies that there exists an invertible formal sheaf $\mathfrak{L}$ on $\mathfrak{X}$ such that $\mathfrak{L}\otimes_{\mathcal{O}_{\mathfrak{X}}}\mathcal{O}_{X}\cong\omega_X$.
\end{proof}

As a consequence of this last proposition, we get that if $X$ is a proper, local complete intersection scheme over a field $k$ and we have a formal deformation $\mathfrak{f}\colon\mathfrak{X}\to\Spf k\llbracket t\rrbracket$, then the relative dualising sheaf $\omega_X$ always extends to the formal deformation $\mathfrak{f}$. To see this, it is enough to observe that l.c.i. schemes/morphisms are in particular Gorenstein schemes/morphisms and then apply the previous proposition.

This result for l.c.i. schemes can be achieved only by using properties of the naive cotangent complex; this second way is described in length in \cite{nobile2021}.

\section{From formal smoothing to geometric smoothing}
In this last section we use all the previous results to show how pass from a formal smoothing to a geometric one. We start by recalling the definition of geometric smoothing.

\begin{definition}\label{def: geometric smoothing}
Let $X$ be a proper scheme. A \emph{geometric smoothing} is a Cartesian diagram
\begin{equation}\label{eq: geometric smoothing diagram}
\begin{tikzcd}
X \arrow[d]\arrow[r, hook] & \mathcal{X} \arrow[d, "\pi"]\\
\Spec k = \Spec \frac{\mathcal{O}_{c, C}}{\mathfrak{m}_{c}} \arrow[r, hook] & C
\end{tikzcd}
\end{equation}
where $C$ is a smooth curve, $c\in C$ is a closed point and $\pi$ is a flat and proper morphism, such that $\pi^{-1}(\eta_C)=:\mathcal{X}_{\text{gen}}$ is smooth, where $\eta_C$ is the generic point of $C$. We say that $X$ is \emph{geometrically smoothable} if it has a geometric smoothing.
\end{definition}

We remark that, if $X$ is smooth over $\Spec k$, then $X$ is geometrically smoothable in a trivial way by considering the trivial family $\text{pr}_2\colon X\times_k C\to C$ of deformations.

We now present some results that will be needed in the proof of the main theorem.

\begin{lemma}[{\cite[Lemma 7.2.1 page 87]{kempf_1993}}]\label{lemma: existence of a curve}
Let $X$ be a scheme, let $U$ be an open, dense subset of $X$ and let $p\in X$ be a closed point. Then there exists an affine curve $C$ in $X$ such that $C$ intersects $U$ and passes through $p$.
\end{lemma}

\begin{remark}\label{remark: completion of local ring at a curve}
Let $C$ be a smooth curve over $k$ and let $c\in C(k)$ be a closed point. Denote by $l$ a local parameter of the maximal ideal $\mathfrak{m}_{c}$ in $\mathcal{O}_{C,c}$. Then there is a isomorphism of topological rings
\[
\widehat{\mathcal{O}_{C,c}}\cong k\llbracket t\rrbracket
\]
such that $l$ is sent to $t$.
\end{remark}

\begin{proposition}\label{prop: morphism from an irreducible and reduced scheme factors trough an irreducible and reduced component of the target}
Let $f\colon X\to Y$ be a morphism of schemes such that $X$ is reduced and irreducible. Then there exists an irreducible and reduced component $Y'$ of $Y$ such that $f$ factors trough $Y'$, i.e. the following diagram commutes
\[
\begin{tikzcd}
X\arrow[rr, "f"] \arrow[rd] & & Y\\
& Y' \arrow[ru, hook]
\end{tikzcd}.
\]
\end{proposition}

\begin{proof}
Since $X$ is irreducible, by \cite[\href{https://stacks.math.columbia.edu/tag/0379}{Tag 0379}]{stacks-project}, $f(X)$ is an irreducible subset of $Y$. Then $Y':=\overline{f(X)}$ is an irreducible component of $Y$ and $f$ factors through $Y'$ by construction. By \cite[II-Ex. 2.3(c)]{hartshorne1977algebraic}, we can always assume $Y'$ to be a reduced scheme.
\end{proof}

\begin{notation}
From now on, we will denote by $\mathfrak{S}$ the formal scheme $\Spf k\llbracket t\rrbracket$ and by $S$ the scheme $\Spec k\llbracket t\rrbracket$. Moreover, for any non-negative integer $n$, we denote by $S_n$ the scheme $\Spec \frac{k\llbracket t \rrbracket}{(t^{n+1})}$.
\end{notation}

The next lemma shows that geometrical smoothability implies formal smoothability.

\begin{lemma}\label{lem: geometric smoothing implies formal smoothing}
Let $X$ be a projective, equidimensional scheme. If $X$ is geometrically smoothable, then it is also formally smoothable.
\end{lemma}

\begin{proof}
Suppose $X$ has a geometric smoothing like \cref{eq: geometric smoothing diagram}, where $c$ is the closed point of $C$ such that the fibre of $\pi$ over $c$ is $X$. Consider the pullback $\tilde{\pi}$ of $\pi$ along the composite morphism $\Spec\widehat{\mathcal{O}_{C,c}}\to\Spec\mathcal{O}_{C,c}\to C$; since $\pi$ is a smoothing of $X$, so is $\tilde{\pi}$. By \Cref{remark: completion of local ring at a curve} we have that the completion of the regular local ring $\mathcal{O}_{C,c}$ is continuously isomorphic to $\mathfrak{S}$. Now using \Cref{remark: constructing formal deformation from classic deformation}, we can construct the associated formal deformation $\mathfrak{p}\colon\mathfrak{X}\to\mathfrak{S}$. We end the argument by invoking \Cref{Tziolas: smoothing over dvr is the same as formal smoothing}.
\end{proof}

At this point we are ready to restate and prove our main result.

\begin{theorem}\label{main result}
Let $X$ be a projective, equidimensional scheme such that one of the following hypotheses hold:
\begin{enumerate}
\item $\coho^2(X,\mathcal{O}_X)=0$,
\item if $X$ Gorenstein, then either the dualising sheaf $\omega_X$ or its dual $\omega_X^{\vee}$ is ample.
\end{enumerate}
Then $X$ is formally smoothable if and only if $X$ is geometrically smoothable.
\end{theorem}

\begin{proof}
One implication is proved in \Cref{lem: geometric smoothing implies formal smoothing}

Suppose we are given a formal smoothing $\mathfrak{p}\colon\mathfrak{X}\to\mathfrak{S}$. Now,
\begin{enumerate}
\item if $\coho^2(X, \mathcal{O}_X)=0$, then by \cite[Theorem~2.5.13]{sernesi2007deformations}, we get that every formal deformation of $X$ is effective; that is to say that there exists a deformation of schemes $p\colon\mathcal{X}\to S$ such that $\mathfrak{X}\cong\hat{\mathcal{X}}_{/X}$. In particular, from the proof, we also deduce that the morphism $p$ is projective.
    
\item By \Cref{dualising sheaf of a Gorenstein proper scheme always extends} the dualising sheaf $\omega_X$ (or $\omega_X^{\vee}$) extends to an invertible formal sheaf $\mathfrak{L}$ on the formal scheme $\mathfrak{X}$. \Cref{Gro: algebrization theorem} then gives us a deformation $p\colon\mathcal{X}\to S$ of $X$ such that the completion of $\mathcal{X}$ along the central fibre is $\mathfrak{X}$. Moreover, as bonus point of the aforementioned theorem, we deduce that $\mathcal{X}$ is projective over $S$.
\end{enumerate}
Concluding, from either hypothesis, if we start with a formal deformation
\[
\begin{tikzcd}
X \arrow[r, hook] \arrow[d] & \mathfrak{X} \arrow[d, "\mathfrak{p}"]\\
\Spf k \arrow[r, hook] & \mathfrak{S}
\end{tikzcd}
\]
then we can construct a deformation of schemes
\begin{equation}\label{eq: deformazione di schemi}
\begin{tikzcd}
X \arrow[d]\arrow[r, hook] & \mathcal{X} \arrow[d, "p"]\\
\Spec k \arrow[r, hook] & S
\end{tikzcd}
\end{equation}
such that $\mathfrak{X}\cong\hat{\mathcal{X}}_{/X}$. Since $\mathfrak{p}$ is assumed to be a formal smoothing and since $k\llbracket t\rrbracket$ is a DVR, we use \Cref{Tziolas: smoothing over dvr is the same as formal smoothing} to conclude that \cref{eq: deformazione di schemi} is a smoothing of $X$. Moreover, in \cref{eq: deformazione di schemi}, the scheme $\mathcal{X}$ is projective over $S$; i.e. there is a non-negative integer $d$ such that $p$ factors as a closed embedding $\iota\colon\mathcal{X}\hookrightarrow\mathbb{P}^d_S=S\times_k\mathbb{P}^d_k$ followed by the first projection $\text{pr}_1\colon\mathbb{P}^d_S\to S$.

Now we use the fact that the Hilbert functor $\mathfrak{Hilb}_{\mathbb{P}^d}$ is representable to deduce the existence of an isomorphism
\[
\alpha_{S}\colon\mathfrak{Hilb}_{\mathbb{P}^d}(S)\to\Hom_{(\text{Sch})}(S,\text{Hilb}_{\mathbb{P}^d}):=\text{h}_{\text{Hilb}_{\mathbb{P}^d}}(S).
\]
Therefore there exists a unique morphism $\psi\colon S\to\text{Hilb}_{\mathbb{P}^d}$ such that both the following diagrams are Cartesian
\[
\begin{tikzcd}
\mathcal{X} \arrow[rr, bend left, "p"] \arrow[r, hook, "\iota"] \arrow[d, "(\id\times\psi)|_{\mathcal{X}}"] & S\times_k\mathbb{P}^d_k \arrow[r, "\text{pr}_1"] \arrow[d, "\psi\times\id"] & S \arrow[d, "\psi"]\\
\text{Univ}_{\mathbb{P}^d} \arrow[rr, bend right, "\text{pr}_1"] \arrow[r, hook] & \text{Hilb}_{\mathbb{P}^d}\times_k\mathbb{P}^d_k \arrow[r, "\text{pr}_1"]& \text{Hilb}_{\mathbb{P}^d}
\end{tikzcd}.
\]
Recall that $\text{Univ}_{\mathbb{P}^d}$ is by definition a closed subscheme of $\mathbb{P}^d_k\times_k\text{Hilb}_{\mathbb{P}^d}$. Inside the Hilbert scheme we consider the smooth locus, defined as follows
\[
\text{H}_{\text{smooth}}:=\left\{[Z]\in\text{Hilb}_{\mathbb{P}^d}(\Spec k)\colon Z\text{ is smooth }\right\}
\]
By \cite[\href{https://stacks.math.columbia.edu/tag/01V5}{Tag 01V5}]{stacks-project}, $\text{H}_{\text{smooth}}$ is an open subset of the Hilbert scheme $\text{Hilb}_{\mathbb{P}^d}$.

Now we study the map $\psi\colon S\to\text{Hilb}_{\mathbb{P}^d}$. 
To do so, we first observe that, since $k\llbracket t\rrbracket$ is a DVR, its spectrum $S$ is made of two points: the closed point, $q$, and the generic point, $\eta$. According to our results so far we have that
\begin{itemize}
\item $\psi(\eta)=[\mathcal{X}_{\text{gen}}]\in\text{H}_{\text{smooth}}$, since (\cref{eq: deformazione di schemi}) is a smoothing;
\item $\psi(q)=[X]\in\text{Hilb}_{\mathbb{P}^d}\setminus\text{H}_{\text{smooth}}$, since $X$ was singular.
\end{itemize}
Since $S$ is connected, there exists a polynomial $\Phi\in\mathbb{Q}[m]$ such that the image of $\psi$ is contained in the connected component $\text{Hilb}_{\mathbb{P}^d}^{\Phi}$ of the Hilbert scheme. By \Cref{prop: morphism from an irreducible and reduced scheme factors trough an irreducible and reduced component of the target} there exists a reduced, irreducible component $Y$ of $\text{Hilb}^{\Phi}_{\mathbb{P}^d}$ such that $\psi$ factors through it:
\[
\begin{tikzcd}
S \arrow[dr, "\widetilde{\psi}"] \arrow[rr, "\psi"] && \text{Hilb}_{\mathbb{P}^d}^{\Phi}\\
& Y \arrow[ur, hook, "i"] &
\end{tikzcd}.
\]
Observe now that if we define $Y_{\text{smooth}}:=Y\cap\text{H}_{\text{smooth}}$ and denote $\overline{Y_{\text{smooth}}}$ the schematic closure of $Y_{\text{smooth}}$, then $\widetilde{\psi}(\eta)\in Y_{\text{smooth}}$ and $\widetilde{\psi}(q)\in\overline{Y_{\text{smooth}}}$. Since $Y_{\text{smooth}}$ is a non-empty open, and therefore dense, subset of $\overline{Y_{\text{smooth}}}$ and $\widetilde{\psi}(q)\in\overline{Y_{\text{smooth}}}$, then we can apply \Cref{lemma: existence of a curve} concluding that there exists a curve $C$ inside $\overline{Y_{\text{smooth}}}$ such that $\widetilde{\psi}(q)\in C$ and $C\cap Y_{\text{smooth}}\neq\emptyset$.

Now let $\nu\colon\tilde{C}\to C$ be the normalisation morphism, and $\tilde{p}\colon\widetilde{\mathcal{X}}\to\widetilde{C}$ be the pullback under the normalisation morphism $\nu$ of the universal family over $Y$. Since $\nu$ is surjective, let $\tilde{c}\in\tilde{C}$ be such that $\nu(\tilde{c})=\widetilde{\psi}(q)$. This completes the proof since we have that the fibre $\tilde{p}^{-1}(\tilde{c})$ is isomorphic to $X$ and $\widetilde{\mathcal{X}}$ is smooth.
\end{proof}

\subsection{Applications of the theorem}
In this section we present an application of our result: smoothability of local complete intersection schemes. We start by recalling the definitions of local complete intersection (l.c.i.) schemes and of complete intersection morphisms.

\begin{definition}\label{def: lci morphism}
Let $f\colon X\to Y$ be a morphism of schemes. We say that $f$ is a local complete intersection morphism, or l.c.i. morphism for short, if it is of finite type and for every point $x\in X$ there are an open neighbourhood $x\in U\subset X$, a scheme $P$ together with a regular immersion $i\colon U\to P$, a smooth morphism of finite type $s\colon P\to Y$ such that $f|_{U}=s\circ i$. We say that a $k$-scheme $X$ is a l.c.i. scheme if the structure morphism $X\to\Spec k$ is a l.c.i. morphism.
\end{definition}

The first remark is that the definition of l.c.i. morphisms does not depend on the factorisation chosen, see \cite[\href{https://stacks.math.columbia.edu/tag/069E}{Tag 069E}]{stacks-project}.

Moreover, if $f\colon X\to Y$ is any morphism of schemes, then the locus $X_{\text{l.c.i.}}$ of points of $X$ such that $f$ is a l.c.i. morphism at $x$, is open in $X$. If we further assume that $f$ is proper, then the locus of points
\[
Y_{\text{l.c.i.}}:=\{y\in Y\colon f\text{ is l.c.i. at }x, \forall x\in f^{-1}(y)\}
\]
is open in $Y$.

\begin{definition}\label{def: complete intersection morphism}
We say that the morphism $f\colon X\to Y$ is a complete intersection morphism if there exists a scheme $P$ together with a global factorisation $s\circ i$ of $f$, with $i\colon X\to P$ a regular immersion and $s\colon P\to Y$ a smooth morphism. We also say that a scheme $X$ is a complete intersection scheme if the structure morphism $X\to\Spec k$ is a complete intersection.
\end{definition}

We now present a theorem of Tziolas \cite[Theorem~12.5]{tziolas2010smoothings} which gives a sufficient condition for the existence of a formal smoothing. We start by introducing the following notation.

\begin{notation}\label{notation: formal neighbourhood Schlessinger sheaf and tangent sheaf}
Let $f\colon X\to Y$ be a morphism of schemes. We denote the relative tangent sheaf by 
$\mathcal{T}_{X/Y}:=\shom_{\mathcal{O}_X}(\Omega^1_{X/Y},\mathcal{O}_X)$ and for $i\in\mathbb{N}$, the \emph{$i^{\text{th}}$ relative cotangent sheaf}\index{$i^{\text{th}}$ Schlessinger's relative cotangent sheaf} in the sense of Schelessinger, see  \cite{LiS1967}, by $\mathcal{T}^i_{X/Y}:=\sext_{\mathcal{O}_X}^{i}(\Omega^1_{X/Y},\mathcal{O}_X)$. In case $Y$ is the spectrum of the ground field $k$, we let $\mathcal{T}_X:=\mathcal{T}_{X/k}$ and $\mathcal{T}^{i}_X:=\mathcal{T}^{i}_{X/k}$ be the tangent sheaf and the $i^{\text{th}}$ cotangent sheaf respectively.
\end{notation}

\begin{theorem}[{\cite[Theorem~12.5]{tziolas2010smoothings}}]\label{Tziolas: existence of formal smoothing}
Let $X$ be a proper, reduced, pure dimensional scheme. If the following conditions hold
\begin{itemize}
\item[(\text{a})] $X$ has complete intersection singularities;
\item[(\text{b})] $\coho^2(X,\mathcal{T}_X)=0$;
\item[(\text{c})] $\coho^1(X,\mathcal{T}^1_X)=0$;
\item[(\text{d})] $\mathcal{T}^1_X$ is finitely generated by its global sections;
\end{itemize}
then $X$ is formally smoothable, i.e. it admits a formal smoothing.
\end{theorem}
As a corollary we would like to mention the following result that can be found in \cite[Corollary~12.9]{tziolas2010smoothings}.

\begin{corollary}\label{coro: lci scheme with normal sheaf satisfy Tziolas conditions}
Let $X$ be a projective, lci scheme such that there exists a regular embedding in a smooth scheme $Y$. If the normal sheaf $\mathcal{N}_{X/Y}$ is finitely generated by its global sections, $\coho^1(X,\mathcal{T}^1_X)=\coho^2(X,\mathcal{T}_X)=0$, then $X$ admits a formal smoothing.
\end{corollary}

Putting together \Cref{coro: lci scheme with normal sheaf satisfy Tziolas conditions} and \Cref{main result} we get the following.

\begin{proposition}
Let $X$ be a singular, projective, l.c.i. variety (i.e. an integral Noetherian scheme of finite type over $k$) over $k$ satisfying conditions $(\text{a})$, $(\text{b})$ and $(\text{c})$ of \Cref{Tziolas: existence of formal smoothing} and such that either its dualising sheaf or its dual is ample. Then $X$ is geometrically smoothable.
\end{proposition}

The above result can be used to get information about points on the moduli space in the following sense.

\begin{theorem}
Let $X$ be a projective, l.c.i. variety with $\omega_{X}$ (respectively $\omega_{X}^{\vee}$) ample. Assume that $X$ satisfies also hypotheses (b), (c) and (d) of \cref{Tziolas: existence of formal smoothing}. Then we have that
\begin{enumerate}
\item $X$ represents a point in closure of the open subset of the (algebraic) moduli stack $\mathcal{M}$ of all projective smooth Gorenstein varieties with ample canonical (respectively anti-canonical) sheaf;
\item the general point of the unique irreducible component of $\overline{\mathcal{M}}$ containing $X$ is smooth.
\end{enumerate}
\end{theorem}

\begin{proof}
The hypothesis of \cref{Tziolas: existence of formal smoothing} and of \cref{main result} are satisfied; hence $X$ is geometrically smoothable. In other words, $X$ represents a point that lies in the closure of the open subset of the moduli stack of projective l.c.i. varieties with ample canonical (respectively anticanonical) sheaf. This proves (1) and (2) above.
\end{proof}

The above theorem has been proved in \cite{fantechi2021smoothing} for the specific case of Godeaux stable surfaces. More precisely, in there the authors verified the hypotheses of Tziolas' \Cref{Tziolas: existence of formal smoothing} and then apply \cref{main result} to show that stable semi-smooth complex Godeaux surfaces appear in the closure of the smooth locus of the moduli stack of stable surfaces of general type and such moduli stack at the point representing the surface has dimension equal to the expected dimension.
\printbibliography

\bigskip
\bigskip
\noindent
\textit{Alessandro Nobile}, \texttt{alessandro.nobile@uni.lu}\\
\textsc{Department of Mathematics, Université du Luxembourg}, 6,  av. de la Fonte, L-4364 Esch-sur-Alzette, Luxembourg
\end{document}

%% file: macros.tex
\usepackage[backend=biber, style=alphabetic, sorting=anyt]{biblatex}
\theoremstyle{definition}
\newtheorem{definition}{Definition}[section]

\newtheorem{example}[definition]{Example}
\newtheorem{remark}[definition]{Remark}
\newtheorem{notation}[definition]{Notation}

\newtheoremstyle{thm} 
        {3mm}
        {3mm}
        {\slshape}
        {0mm}
        {\bfseries}
        {.}
        {1mm}
        {}
\theoremstyle{thm}
\newtheorem{theorem}[definition]{Theorem}
\newtheorem{corollary}[definition]{Corollary}
\newtheorem{lemma}[definition]{Lemma}
\newtheorem{proposition}[definition]{Proposition}
\newtheorem{thm}{Theorem}

\DeclareMathOperator{\Spec}{Spec}

\DeclareMathOperator{\Hom}{Hom}

\DeclareMathOperator{\coho}{H}

\DeclareMathOperator{\shom}{\mathscr{H}\text{\kern -5.3pt {\calligra\large om}}\,\,}

\DeclareMathOperator{\coh}{Coh}

\DeclareMathOperator{\sext}{\mathscr{E}\text{\kern -4.0pt {\calligra\large xt}}\,\,}

\DeclareMathOperator{\id}{Id}
\DeclareMathOperator{\Spf}{Spf}
\DeclareMathOperator{\Fitt}{Fitt}

\usepackage[colorinlistoftodos]{todonotes}